\documentclass[reqno]{amsart}
\usepackage{enumerate}
\usepackage{amsmath,amsthm,amscd,amssymb}
\usepackage{latexsym}
\usepackage{color}
\usepackage{mathrsfs}
\usepackage{float}

\usepackage{bbm} %% for mathbbm: indicator func symbol
\usepackage{mathtools} %% for stackrel: equal in distribution symbol
\usepackage{enumitem} %% for customizing enumeration

\usepackage{calrsfs}
\DeclareMathAlphabet{\pazocal}{OMS}{zplm}{m}{n}

\usepackage{ulem}
\usepackage[all]{xy}
\numberwithin{equation}{section}

\newtheorem{theorem}{Theorem}[section]

\newtheorem{corollary}[theorem]{Corollary}
\newtheorem{lemma}[theorem]{Lemma}
\newtheorem{proposition}[theorem]{Proposition}
\newtheorem{definition}[theorem]{Definition}
\newtheorem{remark}[theorem]{Remark}

\newtheorem{Cond}[theorem]{Condition}

\newcommand{\R}{\mathbb{R}}
\newcommand{\N}{\mathbb{N}}
\newcommand{\Z}{\mathbb{Z}}
\newcommand{\T}{\mathbb{T}}
\newcommand{\E}{\mathbb{E}}
\newcommand{\PB}{\mathbb{P}}
\newcommand{\Y}{\mathbb{Y}}
\newcommand{\OneB}{\mathbbm{1}}

\newcommand{\DD}{\mathcal{D}}

\newcommand{\FF}{\mathcal{F}}

\newcommand{\LL}{\mathcal{L}}
\newcommand{\MM}{\mathcal{M}}

\newcommand{\VV}{V}

\newcommand{\XX}{\mathcal{X}}
\newcommand{\ZZ}{\mathcal{Z}}

\newcommand{\ds}{\displaystyle}

\newcommand{\RNum}[1]{\uppercase\expandafter{\romannumeral #1\relax}}

\newcommand{\e}{\varepsilon}

\renewcommand{\L}{\Lambda}

\renewcommand{\P}{\mathbb{P}}

\renewcommand{\L}{\Lambda}

\renewcommand{\H}{\mathcal{H}}

\newcommand{\Pa}{\mathcal{P}}
\newcommand{\Qb}{\pazocal{Q}}
\newcommand{\Pam}{\mathcal{R}}
\newcommand{\Pb}{\pazocal{P}}

\newcommand{\id}{\mathbbm{1}}

\newcommand{\mf}[1]{{\mathfrak #1}}

\begin{document}

\title{Singular diffusion limit of a tagged particle in zero range processes with Sinai-type random environment}

\begin{abstract}

We derive a singular diffusion limit for the position of a tagged particle in zero range interacting particle processes on a one dimensional torus with a Sinai-type random environment via two steps.  In the first step, a regularization is introduced by averaging the random environment over an $\varepsilon N$-neighborhood. With respect to such an environment, the microscopic drift of the tagged particle is in form $\frac{1}{N}W_\varepsilon'$, where $W_\varepsilon'$ is a regularized White noise.  Scaling diffusively, we find the nonequilibrium limit of the tagged particle $x^\varepsilon_t$ is the unique weak solution of
$d x_t^{\varepsilon} = 2\frac{\Phi(\rho^{\varepsilon}(t, x_t^{\varepsilon}))}{\rho^{\varepsilon}(t, x_t^\varepsilon)} \,W_{\varepsilon}'(x_t^\varepsilon) + \sqrt{\frac{\Phi(\rho^{\varepsilon}(t, x_t^\varepsilon))}{\rho^{\varepsilon}(t, x_t^\varepsilon)}} \,dB_t$,
in terms of the hydrodynamic mass density $\rho^\varepsilon$ recently identified
and homogenized interaction rate $\Phi$.  

In the second step, we show that $x^\varepsilon$, as $\varepsilon$ vanishes, converges in law to the diffusion $x^0$ described informally by
$d x_t^0 = 2\frac{\Phi(\rho^{0}(t, x_t^{0}))}{\rho^{0}(t, x_t^0)} \,W'(x_t^0) + \sqrt{\frac{\Phi(\rho^{0}(t, x_t^0))}{\rho^{0}(t, x_t^0)}} \,dB_t$,
where $W'$ is a spatial White noise and $\rho^0$ is the para-controlled limit of $\rho^\varepsilon$ also recently identified,
solving the singular PDE
$
\partial_t \rho^0 = \frac{1}{2}\Delta \Phi(\rho^0) - 2\nabla \big(W' \Phi(\rho^0)\big)$.

\end{abstract}



\author{Marcel Hudiani}
\address{Mathematics\\
  University of Arizona\\
  621 N. Santa Rita Ave.\\
Tucson, AZ 85750, USA}
\email{{\tt marcelh@arizona.edu}}

\author{Claudio Landim}
\address{IMPA, Estrada Dona Castorina 110, CEP 22460, Rio de Janeiro, Brasil and CNRS
UPRES A 6085, Universit\'e de Rouen, 76801 Mont Saint Aignan Cedex, France.}
\email{{\tt landim@impa.br}}

\author{Sunder Sethuraman}
\address{Mathematics\\
  University of Arizona\\
  621 N. Santa Rita Ave.\\
Tucson, AZ 85750, USA}
\email{{\tt sethuram@arizona.edu}}

\maketitle

%\tableofcontents

\section{Introduction}
\label{sec:micro_dyn}

The problem of characterizing the motion of a tagged or distinguished particle interacting with others has a long history in statistical physics Sections I.8, II.6 in \cite{spohn1991large}, and as well as in the mathematics literature  
Section VIII.4 in \cite{liggett1985interacting},  
and Sections 4.3, 8.4 in \cite{KL}.
Part of the difficulty is that the position $X_t$ of the tagged particle depends on the configuration of the other particles $\xi_t$, and so is not Markovian with respect to its own history.  However, in many settings, one believes that it behaves as a homogenized random walk.
  
Such results with respect to translation-invariant interacting particle systems, in particular mass conserving exclusion and zero-range processes, on $\Z^d$ have been shown when the initial configuration of particles are governed by a stationary or `local equilibrium' distributions.  See \cite{CS}[Section 1.4] for a review with respect to exclusion processes, and \cite{Szr}, \cite{JLS}, \cite{JLS1} for a discussion with respect to zero-range models.    
These descriptions, whether functional law of large numbers or diffusion limits, typically involve the hydrodynamic scaled mass density evolution of the particles in the system.   

Recently, a hydrodynamic limit of a zero-range process in a Sinai-type random environment in a one dimensional torus $\T_N=\Z/N\Z$ was identified in \cite{LPSX}, \cite{FHSX}, \cite{FX} via a two step procedure in terms of a singular, nonlinear stochastic partial differential equation (SPDE).  The aim of this article is to study associated scaling behaviors with respect to a tagged particle in such an inhomogeneous system.  In a nutshell, the limiting diffusion limit will involve the Sinai-type external random environment, as expressed via the hydrodynamic density and other local averages.
A main point is that such a limit formulates a microscopic basis for a class of `singular' Brox-type diffusions, as we will try to explain.

In this sense, our tagged particle results form a natural complement to the hydrodynamics and SPDE convergences in \cite{LPSX}, \cite{FHSX}, \cite{FX}.  These results also generalize the tagged particle diffusion limit in translation invariant zero-range settings without a random environment \cite{JLS}, \cite{JLS1}.

To aid in the description of the results in the article, we first discuss several components in the next subsections.

%===================================
\subsection{Random Environment}

Our environment is built from `Sinai' random environments on  $\Z$.
Namely, consider independent, identically distributed (iid) random variables $\{u_k : k \in \Z\}$ such that $c \leq u_0 \leq 1 - c$ where $0 < c < \frac{1}{2}$ and $E\left[\log\left(\frac{u_0}{1 - u_0}\right) \right] = 0$.  
Let
$U_n$ denote the discrete-time random walk in this random environment (RWRE) with $U_0 = 0$:
\begin{equation}
    P(U_{n+1} = U_n + 1 ~|~ U_n, \{u_k\}) = 1 - P(U_{n+1} = U_n - 1 ~|~ U_n, \{u_k\}) = u_{U_n}
\end{equation}
for $n\geq 1$.
In \cite{sinai1982}, it was proved that $U_n$ scales with the order of $(\log n)^2$. Specifically, with $\sigma^2 = E[(\log (u_0(1-u_0)^{-1}))^2] > 0$, the ratio $\sigma^2 U_n/(\log n)^2$ converges weakly in the annealed sense to a non-trivial random variable $U_\infty$, whose law does not depend on $\sigma$.

A continuous analog of the `Sinai' RWRE was introduced in \cite{brox1986}. Informally, $\sigma$-Brox diffusion is described by the stochastic differential equation (SDE),
\begin{equation*}
    dX^{br}_t = dB_t - \frac{1}{2} W'(X^{br}_t) \,dt ~~,~~ X^{br}_0 = 0.
\end{equation*}
Here, $B, W_+, W_-$ are three independent standard Brownian motions on $\R$, and $W$ is a two-sided Brownian motion:  $W(0) = 0$, $W(x) = \sigma W_+(x)$ for $x > 0$, and $W(x) = \sigma W_-(-x)$ for $x < 0$. 
More rigorously, $\sigma$-Brox diffusion is defined in terms of scale and time-change:
\begin{equation}
    X^{br}_t = s^{-1}(B_{T^{-1}(t)}) ~~\text{where}~~
    \begin{cases}
        \displaystyle s(y) = \int_0^y e^{W(z)} \,dz ~~~,~ y \in \R\\
        \displaystyle T(t) = \int_0^t e^{-2 W(s^{-1}(B_s))}\,ds ~~~,~ t \geq 0.
    \end{cases}
\end{equation}
Analogous to Sinai RWRE, it was shown in \cite{brox1986}, when $\sigma=1$, that $X^{br}_t/(\log t)^2$ converges in distribution to the same limit $U_\infty$.

Interpolating between these two processes,
Seignourel considered in \cite{seignourel2000} environments scaled by $\sqrt{N}$, where $N$ is a scale parameter.  An effective example is $u_k^N = 1/2 + r_k/\sqrt{\sigma^2 N}$ where $\{r_k : k \in \Z\}$ are i.i.d bounded random variables with mean $0$ and variance $\sigma^2$.  
Consider now the diffusively scaled random walk, denoted $U^N_\cdot$, in this array of random environments.   Then, \cite{seignourel2000} showed $U^N_{\lfloor N^2 t \rfloor}/N$ converges weakly in the annealed sense to the $4$-Brox diffusion $X^{br}_t$.

%===================================
\subsection{Hydrodynamics with Zero-range interactions}
\label{subsec:sketch}

The Zero-range process (ZRP) on $\T_N$ follows a collection of continuous time, dependent random walks. If a site $x \in \T_N$ is occupied by $j$ particles, then a particle at $x$ displaces to $y \in \T_N$ at rate $(g(j)/j) \,p(x, y)$ where $g : \N_0 \rightarrow \R_+$ and $p$ is a transition probability. The model has name `Zero-range' because the infinitesimal interaction between particles is with those on the same site.  We mention the motion of `independent' particles is a case when $g(j)\equiv j$.

The model in a `Sinai'-type or `Seignourel' environment on $\T_N$ is when $p(x,x+1) = u^N_x$ and $p(x,x-1) = 1-u^N_x$, with $p(x,y)=0$ for $y\neq x\pm 1$.  Formally, the rate of change of $\eta_t(x)=\xi_{N^2t}(x)$, the number of particles at $x\in \T_N$ in diffusive scale, is given by the generator action
$$N^2L\eta(x) \sim \frac{1}{2}\Delta_N g(\eta(x)) + 2N\nabla^N \Big(g(\eta(x))\frac{r_x}{\sqrt{\sigma^2 N}}\Big)
$$
where $\Delta_NG(x) = N^2(G(x+1)+G(x-1)-2G(x))$ and $\nabla^NG(x) = N(G(x+1)-G(x))$ are the normalized second-order and first-order differences.  Since $r_x=S_x - S_{x-1}$ is the difference of partial sums of $\{r_y\}$, formally $r_x/\sqrt{\sigma^2 N} \sim W(x/N) - W((x-1)/N)$ where $W$ is a spatial standard Brownian motion.  In this way, one can postulate that $\pi^N_t = \frac{1}{N}\sum_{x\in \T_N}\eta_t(x)\delta_{x/N}$ converges to $\rho(t, x)dx$ where
\begin{align}
\label{singular-pde}
\partial_t \rho = \frac{1}{2}\Delta \Phi(\rho) -2\nabla \big(W'(x)\Phi(\rho)\big)
\end{align}
and $\Phi$ is a homogenization of the process rate $g$.  Such an equation, a form of a generalized nonlinear parabolic Anderson model, is a singular SPDE since $\nabla W'\in C^{-3/2-}$ and $\rho\in C^{1/2-}$ (and therefore $\Phi(\rho)\in C^{1/2-}$).

Although this `direct' limit is still open, by considering a two step approach \eqref{singular-pde} was recovered as follows.
 In the first step, after `regularizing' the random environment, that is by replacing $r^N_k/\sqrt{N}$ at a site $k=\lfloor xN\rfloor$ by its average over a small $\varepsilon$-macroscopic block 
 \begin{align}\label{intro_W}
 \frac{1}{2N\varepsilon}\sum_{|i-k|\leq N\varepsilon} r^N_i\sim \frac{1}{2N\varepsilon} \big[S_{k+N\varepsilon}/\sqrt{N} - S_{k-N\varepsilon}/\sqrt{N} \big] \sim \frac{1}{2N\varepsilon}\big[W(x + \varepsilon)-W(x-\varepsilon)\big],
 \end{align} 
 both quenched and annealed hydrodynamic limits are found in \cite{LPSX} with a density $\rho^\varepsilon$ satisfying a regularized form of \eqref{singular-pde}, where $W'$ is replaced with $W_\varepsilon'$, a regularized White noise (cf. Theorem \ref{main thm}, Corollary \ref{main_Cor}). In the second step, the densities $\rho(t,x)=\rho^\varepsilon(t,x)$ as $\varepsilon\downarrow 0$ are shown to converge uniformly in a certain space in probability with respect to $W$ to $\rho^0$, the para-controlled solution of the singular SPDE  \eqref{singular-pde} (cf. Theorem \ref{singular hyd thm}); see \cite{GIP} for an in-depth account of para-controlled distributions.

\subsection{Discussion of results for a tagged particle}
Let $X_t$ be the location of a tagged particle at time $t$ in the system.  The diffusively scaled position $\frac{1}{N}X^N_t = \frac{1}{N}X_{N^2t}$ is approximately given as
$$\frac{1}{N}X^N_{t} - \frac{1}{N}X^N_0 \sim 2\int_0^t \frac{g(\eta_s(X^N_{s}))}{\eta_s(X^N_{s})}W'_\varepsilon(X^N_{s})ds + \frac{1}{N}M^N_{t},$$
in terms of a martingale $\frac{1}{N}M^N_{t}$ with quadratic variation $\int_0^t \frac{g(\eta_s(X^N_{s}))}{\eta_s(X^N_s)}ds$.

Given the quenched hydrodynamics proved in the `first step' with fixed $\varepsilon>0$, one hypothesizes that
$\frac{1}{N}X^N_{t}$ converges to $x^\varepsilon_t$ satisfying the stochastic differential equation (SDE),
\begin{equation}
\label{diff eps}
dx^\varepsilon_t = 2\frac{\Phi(\rho(t, x^\varepsilon_t))}{\rho(t,x^\varepsilon_t)}W'_\varepsilon(x^\varepsilon_t) + \sqrt{\frac{\Phi(\rho(t, x^\varepsilon_t))}{\rho(t,x^\varepsilon_t)}}dB_t.
\end{equation}
And in the `second' step, as $\varepsilon\downarrow 0$, one would hope to obtain that $x_t^\varepsilon$ converges in a suitable sense to $x_t^0$, satisfying formally the SDE with $W'$ replacing $W'_\varepsilon$, and more rigorously given in terms of a scale and time-change functions.  These aims were initially mentioned as an open problem in \cite{LPSX}.

A general form of the first step, where $x^\varepsilon_t$ is a weak solution of \eqref{diff eps}, is proved in Theorem \ref{thm alpha tg} with respect to an abstract disorder $\{\alpha^N_k:k\in \T_N\}$ approximating $\alpha(\cdot)\in C(\T)$.
Then, quenched and annealed limits are found in Corollary \ref{cor-tg-quenched} when $\alpha = W'_\varepsilon$.  For the second step, we state quenched and annealed forms for convergence in distribution of $x^\varepsilon_t$ to the singular diffusion limit $x^0_t$ as $\varepsilon\downarrow 0$ in Theorem \ref{thm singular sde}.  We refer to the diffusion $x^0_t$ as `singular' as it involves the para-controlled solution $\rho^0$ of \eqref{singular-pde} and the multiplicative external noise $W'$.  Derivation of such a `singular' diffusion from microscopic interactions, albeit in two steps, appears to be one of the first of its kind.

We comment, in the case of independent particles, when $g(n)\equiv n$, we have $\Phi(\rho)\equiv \rho$.  The first step limit $x^\varepsilon$ would then satisfy a regularized form of Brox diffusion, $dx^\varepsilon_t = 2W'_\varepsilon(x^\varepsilon_t)dt + dB_t$, while in the second step, the limits of these would be to $4$-Brox diffusion $x^0_t=X^{br}_t$.

In another direction, as mentioned earlier, the results reduce to the tagged particle diffusion limits with respect to translation-invariant Zero-range where $p(x, x\pm 1) = 1/2$, if there is no random environment.  That is, if $\alpha$, $W'_\varepsilon$ and $W'$ were set to $0$, one would recover the time-changed Brownian motion limits in \cite{JLS}, \cite{JLS1}.  

Finally, we comment that the proof methods allow for other disorders and that related singular diffusion limits can be shown.  For instance, in in Remark \ref{bridge rmk} we discuss limits with respect to `Brownian-bridge' disorders.

    \subsection{Proof ideas}
    \label{sec:repl_qty}

    The main idea in the first step (Theorem \ref{thm alpha tg}) is to replace the local function $g(\eta_t(X^N_t))/\eta_t(X^N_t)$ by a function of the mass density $\rho(t, X^N_t)$ at $X^N_t$.  
    The intuition is that at a scaled time $N^2t$, the particles in local neighborhoods have had time to mix: locally, the distribution of particles should be in some sort of equilibrium.   The form of this `equilibrium', since the effects of the random environment are of order $O(1/(N\varepsilon))$ (cf. \eqref{intro_W}), with sufficient mixing estimates, may be approximated by that when the model is translation-invariant, without random environment.
    
    Then, $g(\eta_t(X^N_t))/\eta_t(X^N_t)$ should homogenize to its expected value under an invariant measure $\nu^N_\rho$ for the process $(X^N_t, \eta_t)$ without random environment associated to the local density $\rho=\rho(t, x_t)$.  If we condition on the location of $X^N_t$ and then shift the frame so that $X^N_t$ is at the origin, the homogenization would be in form
$H(\rho) = E_{\nu^0_\rho}\left[ \frac{g(\eta(0))}{\eta(0)}\right]$, the expectation with respect to a `frame' measure denoted $\nu^0_\rho$.  Since the tagged particle is at the origin, $\nu^0_\rho$ can be computed in terms of a size bias with respect to the stationary distribution $\Pam_\rho$ of the process governing indistinguishable particles, $\nu^0_\rho = (\eta(0)/\rho)\Pam_\rho$.  Then,
$H(\rho)$ will have formulation $H(\rho) = E_{\Pam_\rho} \left[ g(\eta(0))/\rho \right]  = \frac{\Phi(\rho)}{\rho}$.

This homogenization, in the presence of the random environment, is made precise in the 
`replacement' Lemma \ref{replacement-lemma}, an important part of the proof of the `first step' result.  The scheme of proof of Theorem \ref{thm alpha tg} and Corollary \ref{cor-tg-quenched} is similar to that in \cite{JLS}, \cite{JLS1} without a random environment.  However, due to the inhomogeneity of the random environment, there are many differences. 

 Indeed, the form of the stationary distribution of the process $(X^N_t, \eta_t)$ involves the structure of the random environment as in the hydrodynamics work \cite{LPSX}. 
It will be important to make local particle number truncations to perform the replacements.   We use `monotone coupling', allowed under an `attractiveness' assumption on $g$ (namely that $g$ is an increasing function), to deduce sufficient truncations, along with estimates on the inhomogeneous stationary distribution in \cite{LPSX} to carry out the homogenizations.  The smoothness of the initial continuum density $\rho_0\in C^\beta(\T)$ for $\beta>0$ assumed in Theorem \ref{thm alpha tg} allows to deduce continuity of $\rho(t,x)=\rho^\varepsilon(t,x)$ (not shown in \cite{LPSX}), important to close the local homogenizations and resulting equations.
We comment in the presence of translation-invariance, and linear growth assumptions on $g$, local particle truncations were avoided in \cite{JLS}, while `attractiveness' was also used in \cite{JLS1} to the treat sublinearity of the rate $g$ assumed there.

From Theorem \ref{thm alpha tg}, one can recover immediately the quenched part of Corollary \ref{cor-tg-quenched} when $\alpha=W'_\varepsilon$.  The annealed part in Corollary \ref{cor-tg-quenched} will follow as a consequence.

To recover quenched and annealed Brox-type diffusion limits $x^0$ in the second step when $\alpha = W'_\varepsilon$ (Theorem \ref{thm singular sde}), we apply the It\^{o}-McKean representation of $x^\varepsilon$ in terms of scale $s^\varepsilon$ and time-change $T^\varepsilon$.  With linear growth bound assumptions on $g$, and additional smoothness of the initial density $\rho_0\in C^{1+\beta}(\T)$ for $\beta>0$, we may verify $\rho=\rho^\varepsilon$ is a classical solution, allowing to plug into the framework of the para-controlled limit Theorem \ref{singular hyd thm}.  Then, with uniform convergence limits of $\rho^\varepsilon$ to $\rho^0$ afforded by the para-controlled limit, we will be able to take limit of $s^\varepsilon$ and $T^\varepsilon$ to $s^0$ and $T^0$ as $\varepsilon\downarrow 0$.  In the end, the It\^{o}-McKean form of the limit describes the diffusion $x^0$ in Theorem \ref{thm singular sde}.

%%%%%%%%%%%%%%%%%%%%%%%%%%%%%%

\medskip

\noindent {\bf Plan of the article.} After specifying more carefully the model in Section \ref{sec:zrp} and relevant results and consequences in the literature, we turn to our results for a tagged particle (Theorem \ref{thm alpha tg}, Corollary \ref{cor-tg-quenched}, and Theorem \ref{thm singular sde}) in Section \ref{ch:main_thm}.  In Section \ref{sec:varep>0}, we give the proof of Theorem \ref{thm alpha tg}, assuming a replacement Lemma \ref{replacement-lemma}.  In Section \ref{sec:varep-to-0}, we show Theorem \ref{thm singular sde}.  Finally, in Sections \ref{sec:loc_1b}, \ref{sec:loc_2b}, and \ref{sec:glob_repl}, we complete the proof of Lemma \ref{replacement-lemma} in three parts via `local 1-block', `local 2-block', and  `global replacement' estimates respectively.

\section{Model description}
\label{sec:zrp}

Let $\T_N:=\Z/N\Z$ be the discrete torus for $N\in\N$.  Throughout
this article, we will identify $\T_N$ with $\{1,2,\ldots,N\}$ and also
identify the continuum unit torus $\T$ with $(0,1]$. 
Consider a deterministic `environment' on the discrete torus
$\{\alpha_k^N : k\in \T_N\}$ such that their linear interpolations for
$u\in \T$,
$$Y_u^N = \alpha_{\lfloor Nu\rfloor}^N + (Nu- \lfloor Nu\rfloor) \alpha_{\lfloor Nu\rfloor +1}^N,$$
converge uniformly to $\alpha(u)$ where $\alpha(\cdot)$ is a
continuous function on $\T$.  

We now introduce the zero-range
process on $\T_N$ with respect to this environment.
Let $\N_0 = \{0,1,2,\ldots\}$, and let
$\Sigma_N = \N_0^{\T_N}$ be the (countable) configuration space of particles. 
For a particle configuration $\xi\in\Sigma_N$, the coordinate $\xi(k)$ for
$k\in \T_N$ denotes the number of particles at site $k$.

With respect to a function $g:\N_0 \to \R_+$, denote by $\xi_t$ for $t\ge 0$ the
continuous-time Markov chain, informally
described as follows. Since $\max_{1\leq k\leq N} |\alpha_k^N|$ is uniformly bounded in $N$, take $N$ sufficiently large so that
$|\, \alpha_k^N\,|/N <1/2$ for all $k\in\T_N$. At each location $k\in \T_N$, a clock rings at rate $g(\xi_t(k))$, at which time a particle is selected at random from those at $k$ to move to $k\pm 1$ with probability $(1/2) \pm (\alpha_k^N/N)$.  The case $g(j)\equiv j$ corresponds to when all the particles in the system are independent, each carrying its own exponential rate $1$ clock.

More precisely, what we call the `standard' process $\{\xi_t: t\geq 0\}$ is the Markov
continuous time jump process on $\Sigma_N$, with
generator $L$ given by
\begin{align}
\label{eqn: generator L}
Lf(\xi) &=
\sum_{k\in \T_N} 
\Big\{
g(\xi(k)) \Big( \dfrac12 + \dfrac{\alpha_k^N}{ N}\Big)
\big(f(\xi^{k,k+1}) - f(\xi) \big)\nonumber\\
&\quad \quad +
g(\xi(k)) \Big( \dfrac12 - \dfrac{\alpha_k^N}{ N}\Big)
\big(f(\xi^{k,k-1}) - f(\xi) \big)
\Big\}.
\end{align}
Here, $\xi^{j,k}$ is the configuration obtained from $\xi$ by moving a
particle from $j$ to $k$, that is,
\begin{equation*}
\xi^{j,k} (\ell) \;=\;
\begin{cases}
\xi(j) -1 & \ell=j\;, \\
\xi(k)+1 & \ell=k \;, \\
\xi(\ell) & \ell\not = j\,, k \;.
\end{cases}
\end{equation*}

Our main focus will be the behavior of a tagged or distinguished particle in the system.  
Let $X_t$ denote the location of the tagged particle in $\T_N$ at time $t\geq 0$. Since the dynamics of a particle depends on the location of the other particles, the process $X_t$ by itself is not Markovian. 
However, we may consider the coupled process $(X_t, \xi_t)$ on $\big\{(x,\xi)\in \T_N\times \Sigma_N: \xi(x)\geq 1\big\}$ with respect to the deterministic environment $\{\alpha_x^N\}$.  Such a process is Markovian with generator
\begin{align}
 \label{eq:gen}
        \LL_N f(x, \xi) &= \sum_{\pm} \sum_{y \in \T_N \setminus \{x\}} 
 \Big\{ g(\xi(y))   p_y^{N,\pm}(f(x, \xi^{y,y\pm 1}) - f(x, \xi)) \\
        &\quad\quad + g(\xi(x)) \frac{\xi(x) - 1}{\xi(x)}  p_x^{N,\pm}(f(x, \xi^{x,x\pm 1}) - f(x, \xi))\nonumber\\
        &\quad \quad + g(\xi(x)) \frac{1}{\xi(x)}  p_x^{N,\pm}(f(x\pm 1, \xi^{x,x\pm 1}) - f(x, \xi)) \Big\},\nonumber
\end{align}
 where
$p^{N,\pm}_x := \frac{1}{2} \pm \frac{\alpha^N_x}{N}$.

Of course, $\mathcal{L}_N$ restricted to functions $f(x,\xi) = f(\xi)$ only of the configuration $\xi$ reduces to $L$.

We also observe the restriction of $\LL_N$ to functions $f(x,\xi)=f(x_0, \xi)$, for a fixed $x_0\in \T_N$, is itself a generator of the process $\xi_t$ where $\xi_t(x_0)\geq 1$ and the tagged particle does not move, but is always at $X_0=x_0$:
\begin{align*}
       & \LL^{env, x}_N f(x_0,\xi) = \sum_{y \in \T_N \setminus \{x_0\}} g(\xi(y))  \left[ p_y^{N,+}(f(x_0,\xi^{y,y+1}) - f(x_0,\xi)) + p_y^{N,-}(f(x_0, \xi^{y,y-1}) - f(x_0, \xi)) \right]\nonumber\\
        &\quad + g(\xi(x_0)) \frac{\xi(x_0) - 1}{\xi(x_0)} \left[ p_{x_0}^{N,+}(f(x_0, \xi^{x_0,x_0+1}) - f(x_0, \xi)) + p_{x_0}^{N,-}(f(x_0,\xi^{x_0,x_0-1}) - f(x_0,\xi)) \right].
        \end{align*}

For each $N$, we will later observe the evolution of the tagged particle and zero-range process
when time is speeded up by $N^2$.
We define processes
$$X^N_t := X_{N^2t} \ \ {\rm and \ \ } \eta_t: = \xi_{N^2t},$$
generated by $N^2\LL_N$ 
for times $0\leq t\leq T$, where $T>0$ can be any
fixed time horizon. 

\subsection{Assumptions on $g$}
\label{g assump section}
To avoid degeneracies, we suppose that $g(0)=0$ and $g(k)>0$ for $k\geq 1$.   We now state further conditions on $g$ that we will use in various combinations in the main theorems:

\begin{enumerate}
\item[(A)] $g(k+1)\geq g(k)$ for $k\geq 0$;
\item[(LG)]
$\sup_{k\in \N}|g(k+1) - g(k)| \leq g^* <\infty$;
\item[(M)] There is an $m\geq 1$ and $a_0>0$ such that $g(k+m)-g(k)\geq a_0>$ for $k\geq 0$.
\end{enumerate}

We mention how these properties are
used in several places to make estimates.

Condition (A), often called `attractiveness' allows a `basic coupling'
 (cf. \cite{Andjel}, \cite{liggett1985interacting}, Theorem II.5.2 in \cite{KL}):  If $dR$ and
$dR'$ are initial measures of two standard processes $\xi_0$ and $\xi'_0$ at time $t=0$, and
$dR\ll dR'$ in stochastic order, then the distributions of $\xi_t$ and
$\xi'_t$ at times $t\geq 0$ are similarly stochastically ordered. We will primarily use (A) to 
bound above and below mean local particle densities at sites, and to make `$\log(N)$' site particle truncations Lemma \ref{loc trun lemma}
in the $1$ and $2$-blocks replacement estimates
Lemmas \ref{lem:L1b} and \ref{lem:L2b}.

 Condition (LG) is used in several places to bound $g(j)\leq g^*j$, and also in the $1$ and $2$ block estimates.
   Condition (M) is used to deduce $g(j)/j\geq g_*$ for a positive constant $g_*$.  Such bounds are used to plug into assumptions for the singular hydrodynamic limit Theorem \ref{singular hyd thm}.  One also deduces (LG) and (M) give that $j\geq 1\mapsto g(j)/j$ is Lipschitz:  $|g(j)/j - g(m)/m| \leq |g(j)-g(m)|/j +g(m)|j-m|/(mj) \leq 2(g^*/j)|j-m|\leq 2g^*|j-m|$.
   
   Importantly, also in combination, (LG), (M) imply a mixing property of the process:
Let
$b_{l,j}$ be the
inverse of the spectral gap of the standard process with reflecting boundary conditions, in a null environment that is when $\alpha^N_k\equiv 0$, defined on the cube
$\Lambda_l = \{-l,\ldots, l\}$ with $j$ particles; see Section
\ref{spec_gap section} for precise definitions.
Under (LG) and (M) we have
\begin{equation}
\label{spec_gap_condition}
b_{l,j} = O(l^2),
\end{equation}
 uniformly in $j\geq 0$ \cite{LSV}.  
Such a bound
 is used to compare the inhomogeneous system in the environment $\{\alpha^N_k\}$ with a translation-invariant system, when $\alpha^N_k\equiv 0$.

 We comment that, although \cite{JLS} also assumed (LG) and (M), a more general class of $g$'s might be considered.  Especially, (M) might be relaxed so that bounded or sublinear rate processes are allowed for the regularized tagged particle limit (Theorem \ref{thm alpha tg}) as in the work \cite{JLS1}.  However, (LG), (M) give that $g_*j\leq g(j)\leq g^*j$ and therefore $g_*\leq \Phi(\rho)/\rho \leq g^*$ (cf. \eqref{Phi_eqn}), an assumption in \cite{FHSX}, \cite{FX}.  Since these citations play a role in the singular diffusion limit Theorem \ref{thm singular sde}, to have a unified set of assumptions, we have specified the class as above.

%%%%%%%%%%%%%%%%%%%%%%%%%%%%%%%%%%%%%%%%%%%%%%%%%%%%

\subsection{Quenched random environment formulation}
\label{quenched-section}
Let $\left\{ r_x \right\}_{x\in\N}$ be a sequence of i.i.d.\,random
variables with mean $0$ and variance $0<\sigma^2 <\infty$.
Let $s_0 = 0$ and for $n\geq 1$, let $s_n = \sum_{k=1}^n r_k $. For
$0\leq u \leq 1$, let
\begin{equation*}
Y^N_u =
 \dfrac{1}{\sigma \sqrt N} 
 s_{\lfloor Nu \rfloor}
+
\dfrac{Nu - \lfloor Nu \rfloor} {\sigma \sqrt N} r_{\lfloor Nu \rfloor
  + 1}\;, 
\end{equation*}
where $\lfloor a \rfloor$, $a\in \R$ stands for the integer part
of $a$.  

It is standard that the random functions $\left\{Y^N_u: 0\leq u\leq 1\right\}$ converge in distribution as $N\uparrow\infty$ to the
Brownian motion on $[0,1]$.  By Skorokhod's Representation Theorem, we
may find a probability space $(\Omega, \FF, \Pa)$ and
$\left\{W^N_u: 0\leq u\leq 1\right\}$, $N\in\N$, mappings from
$\Omega$ to $C[0,1]$, such that, for all $N\in\N$,
$$\left\{Y^N_u: 0\leq u\leq 1\right\}=\left\{W^N_u: 0\leq u\leq
1\right\}$$ in distribution and moreover,
$\left\{W^N_u: 0\leq u\leq 1\right\}$ converges uniformly almost
surely to the standard Brownian motion
$\left\{W_u, 0\leq u\leq 1\right\}$.

 We will now fix an $\omega\in \Omega$ such that 
$\left\{W^N_u(\omega): 0\leq u\leq 1\right\}$ converges (uniformly) to a Brownian path $\left\{W_u(\omega): 0\leq u\leq 1\right\}$.

Following the formulation in \cite{LPSX}, we extend $W_u^N$ as well as $W_u$ to
$u\in[-1,2]$. With $\tilde W_u$ representing either $W^N_u$ or $W_u$,
define 
\[
\tilde W_u =
\begin{cases}
\tilde W_{u+1} - \tilde W_1& u\in [-1,0),\\
\tilde W_{u-1} + \tilde W_1& u\in (1,2].
\end{cases}
\]
Here, in the time intervals $[-1,0)$ and $(1,2]$, the trajectory $\tilde W$ starts respectively from $-\tilde W_1$ and $\tilde W_1$, then displaces according to $\tilde W$ in $[0,1]$.  In this way, the increments of $\tilde W$ are periodic in $\T$.  To simplify notation, we will drop the tilde in the notation for $\tilde W$.

Let $\varepsilon>0$ be a parameter.  Let
also $\psi:[-1,1] \mapsto \R$ be a $C^1$ function with $\int_{-1}^1\psi(x)dx =1$.  For each $N\in \N$ and
$k\in \T_N$, consider an $\e$-regularization of local environments
such that
$$
q_k^N \;\stackrel{d}{=} \;
\frac{1}{\sigma N}
\sum_{|j-k|\leq \lfloor  N\varepsilon \rfloor} r_j\,
\Big[\, \frac{1}{\varepsilon}
\psi \big(\frac{j-k}{N\varepsilon}\big)\, \Big],
$$
namely
\begin{align}
\label{eqn: q and W}
q_k^N &:= \frac{1}{\sqrt{N} \varepsilon}
\sum_{j =  k - \lfloor  N\varepsilon \rfloor}
^{k+ \lfloor  N\varepsilon  \rfloor -1}
W^N_j \, \Big[\, \psi\big(\frac{j-k}{N\varepsilon}\big)
- \psi\big(\frac{j-k+1}{N\varepsilon}\big)\, \Big] \nonumber \\
&\qquad\quad + \;\frac{1}{\sqrt{N}\varepsilon}
W^N_{k+N\varepsilon}\,
\psi \big( \frac{\lfloor  N\varepsilon \rfloor}{N\varepsilon} \big) -
\frac{1}{\sqrt{N}\varepsilon} W^N_{k-1-N\varepsilon}
\psi \big( -\, \frac{\lfloor  N\varepsilon \rfloor}{N\varepsilon} \big) \;.
\end{align}
In particular, when $k/N\rightarrow x\in \T$ as $N\uparrow\infty$, we
have $\sqrt{N} q_k^N$ converges to
\begin{align}
\label{q_conv}
&-\frac{1}{\varepsilon^2}\int_{x-\varepsilon}^{x+\varepsilon} W(u) \psi'\big(\frac{u-x}{\varepsilon}\big)du + \frac{1}{\varepsilon}\big\{ W(x+\varepsilon)\psi(1) - W(x-\varepsilon)\psi(-1)\big\}\\
& = \frac{d}{dx}W * \psi_\varepsilon(x) =: W'_\varepsilon(x), \nonumber
\end{align}
where $\psi_\varepsilon(u) = \frac{1}{\varepsilon}\psi(u/\varepsilon)$.

As $W^N_\cdot$ converges uniformly to $W_\cdot$, by the continuity
of $W$, and the properties of $\psi$, in view of \eqref{eqn: q and W},
there exists a constant $C=C(\omega)<\infty$ such that
\begin{equation*}
\limsup_{N\rightarrow\infty}
\max_{1\leq k\leq N} \left\{\sqrt N \, |q_k^N| \right\} \leq C/\varepsilon\;.
\end{equation*}

As remarked in the introduction, $W'_\varepsilon$ is a smoothing of $W'$.  In particular, $W'_\varepsilon\in C^\gamma$ when $\psi'\in C^\gamma$ for $\gamma\geq 1/2$ and $\psi(\pm 1)=0$.  When one of $\psi(1)$ or $\psi(-1)$ does not vanish, then $W'_\varepsilon\in C^{1/2-}:= \cap_{0<\epsilon< 1/2} C^{1/2-\epsilon}$.  Also, when $\psi'\in C^\gamma$ for $\gamma<1/2$, then $W'_\varepsilon \in C^{1/2-}$. 
 A natural case is when $\psi(x) = (1/2) \id_{[-1,1]}(x)$ for which
$W'_\varepsilon(x) = (2\varepsilon)^{-1}[W(x+\varepsilon) -
W(x-\varepsilon)]$. 

 Here, and for later use, we denote by $C^\gamma(\Y)$ for $\gamma\geq 0$, and by $C^{\gamma_1,\gamma_2}(\Y)$ for $\gamma_1,\gamma_2\geq 0$, the standard H\"older spaces of functions on respective spaces $\Y$.

In this article, with respect to the above quenched setting, we focus on
the deterministic sequence 
\begin{align}
\label{alpha defn}
\alpha_k^N \equiv \sqrt{N}q_k^N \ \ {\rm  and\ \ }
\alpha(u) \equiv W'_\varepsilon(u),
\end{align}
although other periodic sequences, such as $a_k^N \equiv \alpha_k^N - \tfrac{1}{N}\sum_{j\in \T_N}\alpha_j^N$ and the limit $a(u) \equiv W'_\varepsilon(u) -\int_\T W'_\varepsilon(v)dv$, which corresponds to Brownian-bridge random environments, could be considered.

\subsection{Invariant measures} \label{subsec: invariant measure} We first consider invariant measures for $L$, generating the standard process of indistinguished particles.  Then, we consider invariant measures for the system with a tagged particle generated by $\LL_N$.
In the following, with respect to a given probability measure $\mu$,
we denote by $E_\mu$ and ${\rm Var}_\mu$ its expectation and variance.

The
building block for the invariant measures of $L$ is $\{\Pb_{\phi}\}$, a
family of Poisson-like distributions indexed by `fugacities' $\phi\geq 0$. For each $\phi$, the measure $\Pb_{\phi}$ on $\N_0$ is
defined by
\[
\Pb_{\phi}(n) =
\dfrac{1}{\ZZ(\phi)} \dfrac{\phi^n} {g(n) !},
\quad \text{for } n\geq 0,
\]

where
\begin{equation*}
    \ZZ(\phi) = \sum_{k=0}^\infty \frac{\phi^k}{g(k)!}
    ~~\text{and}~~
    g(n)! = g(n) \, g(n - 1) \cdots g(1) ~,~ g(0)! = 1.
\end{equation*}
The family $\{
\Pb_\phi: \phi \in [0, \phi^{*})\}$ is well defined
where $\phi^{*}$ is the radius of convergence for $\ZZ$. Under condition (M), $\phi^{*}=\lim_{j\uparrow\infty}g(j)=\infty$.  Hence, $\Pb_\phi$ satisfies the (FEM) condition in p. 69 of \cite{KL}.

Let $R(\phi) = E_{\Pb_\phi} [ Z ]$, where $Z(n) = n$, be the mean of
the distribution $\Pb_\phi$. A direct computation yields
$R'(\phi)>0$, $R(0)=0$ and, as $\phi^*=\infty$, we have
$\lim_{\phi\to\phi^*} R(\phi) = \infty$. 
Since
$R$ is strictly increasing, it has an inverse, denoted by
$\Phi: \R_+ \to \R_+$, which is also strictly increasing with $\Phi(0)=0$.  We may parametrize the family of
distributions $\Pb_{\phi}$ by its mean. For $\rho\geq 0$, let
$\Qb_\rho = \Pb_{\Phi(\rho)}$, so that
$E_{\Qb_\rho} [ Z] = E_{\Pb_{\Phi(\rho)}} [ Z ] = R(\Phi(\rho)) =
\rho$.

A straightforward computation yields that
$E_{\Pb_\phi} [ g(Z) ] = \phi$ for $0\leq \phi\leq \phi^*=\infty$. Thus,
\begin{equation}
\label{Phi_eqn}
\Phi(\rho) = E_{\Pb_{\Phi(\rho)}} [ g(Z) ]
= E_{\Qb_{\rho}} [ g(Z) ]\ \;, \quad \rho\,\ge\, 0\;.
\end{equation}
Recall, under condition (LG) that $g(k) \leq g^*k$, and therefore
$\Phi(\rho) \leq g^*\rho$.  Under (M), recall there is $g_*>0$ such that $g(k)\geq g_*k$, and so $\Phi(\rho)\geq g_*\rho$.  A simple
computation yields that $\Phi'(\rho) = \Phi(\rho)/\sigma^2(\rho)$
where $\sigma^2(\rho)$ is the variance of $Z$ under $\Qb_\rho$.
Moreover, one may compute that $\Phi\in C^\infty$ is a smooth function.

We note, in the case $g(k)\equiv k$, that $\Phi(\rho)\equiv \rho$, and $\Pb_\phi$ is a Poisson measure with mean $\phi$. 

Fix a vector $( \phi_{k,N} : k\in\T_N)$ of non-negative real numbers.
Denote by $\Pam_N=\Pam_N(\cdot;
 \{\phi_{k,N}\})$ the product measure on $\N_0^{\T_N}$ whose
marginals are given by
\begin{equation}
\Pam_N (\xi(k) = n) = \Pb_{\phi_{k,N}}(n),
\quad \text{for } k\in \T_N, n\geq 0.
\label{star marginal}
\end{equation}
It is straightforward (cf.~\cite{Andjel}) to check that $\Pam_N$
is invariant with respect to the generator $L$ in \eqref{eqn:
  generator L} as long as the fugacities $\{\phi_{k,N}\}_{k\in \T_N}$
satisfy:
\begin{equation} \label {eqn: invariant phi}
\Big( \dfrac12 + \dfrac{\alpha_{k-1}^N}{ N}\Big)\phi_{k-1,N}
+
\Big( \dfrac12 - \dfrac{\alpha_{k+1}^N}{ N}\Big)\phi_{k+1,N}
=
\phi_{k,N},
\quad
k=1,2,\ldots,N.
\end{equation}

Notice that $\{c \phi_{k,N}\}_{k\in \T_N}$, for $c\geq 0$, is a solution
of \eqref{eqn: invariant phi} if $\{\phi_{k,N}\}_{k\in \T_N}$ is a
solution.  In particular, any solution gives rise to a one-parameter
family of solutions $\Pam_{N,c} =\Pam_N(\cdot; \{c\phi_{k,N}\})$ for $c\geq 0$. In Lemma 2.1 in \cite{LPSX}, 
it is shown that \eqref{eqn: invariant phi} admits a solution, unique up to multiplicative constant, that is strictly positive.

We now restate the following useful estimates of the `fugacities' given in Lemma 2.2 in \cite{LPSX}.
Let
$ \phi_{\max,N} = \max_{1\leq k\leq N} \left\{\phi_{k,N} \right\}$
and
$ \phi_{\min,N} = \min_{1\leq k\leq N} \left\{\phi_{k,N}
\right\}$.

\begin{lemma} 
\label {lem: uniform bounds on phi max min}
Let $\{\phi_{k,N}\}_{k\in \T_N}$ be a solution of \eqref{eqn: invariant phi}. Then,
there exist constants $C_1, C_2<\infty$ such
that for all $N\in \N$
\[
1\leq  \dfrac{\phi_{\max,N}}{\phi_{\min,N}} \leq C_1\quad
 \text{and} \quad
  \max_{1\leq k\leq N} |\phi_{k,N} - \phi_{k+1,N}|
 \leq \frac{C_2}{N}\, \phi_{\max,N}.
\]
As a consequence, if $\phi_{\max, N}=C$ then $\phi_{\min, N}\geq CC_1^{-1}>0$.
\end{lemma}

  Let $\{\rho_{k,N} =E_{\Pb_{\phi_{k,N}}}[Z]\}_{k\in \T_N}$ be the mean values where $\{\phi_{k,N}\}_{k\in \T_N}$ satisfies \eqref{eqn: invariant phi}.  Since by Lemma \ref{lem: uniform bounds on phi max min}, the parameters $\phi_{k,N}$ are uniformly bounded above and below, we have that $\rho_{k,N}$ is also uniformly bounded above and below.

We now turn to stationary measures for the process $(X_t, \xi_t)$ generated by ${\mathcal L}_N$.

\begin{proposition}
    The measure $\nu^N$ on  $\T_N \times \Sigma_N$ given by
    \begin{equation}
    \label{nu N def}
        \nu^N(x, \xi) = \frac{\xi(x)}{ \sum_{y \in \T_N} \rho_{y,N}} \, \Pam_N(\xi)
    \end{equation}
is invariant for the joint process generated by $\mathcal{L}_N$.
Moreover, for each fixed $x_0\in \T_N$, the measure $\nu^{env,x_0}$ on $\Sigma_N$ given by $\nu^{env,x_0}(\xi)=\nu^N(\xi|X=x_0)=\frac{\xi(x_0)}{\rho_{x_0,N}}\Pam_N(\xi)$ is invariant for the process generated by $\mathcal{L}^{env, x_0}_N$.

    \label{prop:nu_stat_meas}
\end{proposition}

We comment, analogous to the invariant measures of $L$, since the parameters $\{\phi_{k,N}\}_{k\in \T_N}$ when multiplied by a constant $c$ still satisfy \eqref{eqn: invariant phi}, the joint process generated by $\mathcal{L}_N$ also has a family of invariant measures indexed by $c\in \R$.  

The proof of Proposition \ref{prop:nu_stat_meas} is a straightforward but long computation.  In particular, we may compute the $L^2(\nu^N)$ adjoint $\mathcal{L}^*_N$ as
\begin{align*}
\mathcal{L}^*_Nf(x,\xi) &= \sum_{y\neq x}\sum_\pm \big (f(x, \xi^{y,y\pm1})-f(x,\xi)\big)g(\xi(y))p^{N,\mp}_{y\pm 1} \frac{\phi_{y\pm 1, N}}{\phi_{y,N}}\\
& \quad \quad +\sum_{\pm}\big(f(x, \xi^{x,x\pm 1})-f(x,\xi)\big) g(\xi(x))\frac{\xi(x)-1}{\xi(x)} p^{N,\mp}_{x\pm 1, N} \frac{\phi_{x\pm 1, N}}{\phi_{x,N}}\\
& \quad \quad +\sum_{\pm} \big(f(x\pm 1, \xi^{x, x\pm 1})-f(x,\xi)\big) \frac{g(\xi(x))}{\xi(x)} p^{N,\mp}_{x\pm 1} \frac{\phi_{x\pm 1, N}}{\phi_{x,N}}.
\end{align*}
We also find the $L^2(\nu^{env, x_0})$ adjoint of $\mathcal{L}^{env, x_0}$, denoted $\mathcal{L}^{*,env, x_0}_N$, is the operator restricted to functions $f(x,\xi) = f(x_0, \xi)$ where $X=x_0$ is fixed.  The relations $\mathcal{L}^*_N 1 = \mathcal{L}^{*,env, x_0} 1 = 0$ imply invariance of $\nu^N$ and $\nu^{env, x_0}$.

We comment in passing, as shown by a straightforward computation, that $\nu^N$ is reversible, that is $\mathcal{L}_N = \mathcal{L}_N^*$, exactly when $\sum_{x\in \T_N} \alpha^N_x = 0$.  For instance, the Brownian bridge sequence $a_k^N$ and $a(u)$ given near \eqref{alpha defn} would satisfy this condition.

\subsection{Invariant measures for the translation-invariant process}
\label{stat measures section}
When $\alpha^N_\cdot\equiv 0$, that is in the translation-invariant setting, $\phi_{k,N}\equiv \phi$ is constant in $k$.  Let $\rho$ be the mean of $\Pb_{\phi}$.  Then, we may write the fugacity $\phi=\Phi(\rho)$. 
In this case,
 the stationary measure $\nu^N$ of $(X_t, \xi_t)$ reduces to 
 the measure denoted as
$$\nu_{\rho}^N(x, \xi) = \frac{\xi(x)}{N\rho}\Pam_\rho,$$
where $\Pam_{\rho} = \prod_{k\in \T_N} \Pb_{\Phi(\rho)}$.  Note that the conditional measure, denoted $\nu^0_\rho$ in the introduction, satisfies $\nu^0_\rho = \nu_\rho^N(\xi\in \cdot| X=0) = \frac{\xi(0)}{\rho}\Pam_\rho(\xi\in \cdot)$.
We comment that the measures $\{\Pam_\rho\}_{\rho\geq 0}$ are well-known as stationary distributions of the translation-invariant process $\xi_\cdot$ governed by $L$ when $\alpha^N_\cdot\equiv 0$ (cf. \cite{JLS}, \cite{JLS1}).

\subsection{Initial measures}

We specify now conditions on the initial measures $\mu^N$ for the zero-range
process with a tagged particle $(X_t, \xi_t)$ generated by $\mathcal{L}_N$.
Denote by $\mu^N_X$ and $\mu^N_L$ the marginals with respect to $X$ and $\xi$ on spaces $\T_N$ and $\Sigma_N$ respectively.

We now fix $\Pam_N$
(cf. \eqref{star marginal}) to be the invariant measure for $\xi_t$ chosen so that
$\phi_{{\rm max}, N}= \max_{k\in\T_N} \left\{\phi_{k,N} \right\}=1$.  Such a choice specifies the normalization or parameter $c$ multiplying a solution $\phi_{k,N}$.  We comment other parameter values could be also be used.  Since the fugacities $\phi_{x,N}\leq \phi_{max,N}=1$, the maximum density $\rho_{max, N}$ (corresponding to $\phi_{max, N}$) is bounded uniformly in $N$.  

Let $\nu^N$ be the associated invariant measure defined in terms of $\Pam_N$ (cf. \eqref{nu N def}). 
Define the relative entropy between measures $\mu_1$ and $\mu_2$ by $\H(\mu_1 | \mu_2):= \int f \ln f \,d\mu_2$ where $f = d\mu_1/d\mu_2$.

\begin{Cond} \label{cond: initial measure} 
The measure $\mu^N$ satisfies the following.
\begin{enumerate}
\item[(a)] We have $\{\mu_L^N\}_{N\in \N}$ is associated with an initial density profile $\rho_0\in L^1(\T)$ for $\beta\geq 0$ in the
sense that for any $G\in C(\T)$ and $\delta>0$
\begin{equation*}
\lim_{N\to \infty} \mu_L^N
\Big[ \, \Big|
\dfrac1N \sum_{k\in \T_N} G\Big(\dfrac kN \Big) \xi(k)
-
\int_\T G(x) \rho_0(x) dx
\Big|
>\delta
\, \Big]
 =0\; .
\end{equation*}

\item[(b)] The relative entropy of $\mu_L^N$ with respect to $\Pam_N$ is
of order $N$. That is, there exists a finite constant $C_0$ such that
$\H(\mu_L^N | \Pam_N)\le C_0 N$ for all
$N\ge 1$.
\item[(c)] We have $\{\mu_X^N\}_{N\in \N}$ converges weakly to the law of a random variable $Z_X$ on $\T$.

\item[(d)] The relative entropy of $\mu^N$ with respect to $\nu^N$ is of order $N$:  There is a finite constant $C_0$ such that $\H(\mu^N|\nu^N) \leq C_0N$ for all $N\geq 1$.
\item[(e)] The marginal $\mu_L^N$ is stochastically bounded above 
and below, 
$\Pam_{N, c_0}\ll\mu_L^N\ll
 \Pam_{N, c_1}$ 
for $0<c_0<c_1<\infty$.  That is, for any increasing coordinatewise function $f:\Sigma_N\rightarrow \R$, we have
$
E_{\Pam_{N,c_0}}[f]\leq 
E_{\mu^N_L}[f]\leq E_{\Pam_{N,c_1}}[f]$.
As a consequence, by Lemma \ref{lem: uniform bounds on phi max min} and properties of $\Pb_\phi$, there is $0<\rho_-=\rho_-(c_0)<\rho_+(c_1)=\rho_+<\infty$ such that
$\Pam_{\rho_-}\ll \Pam_{N, c_0}$ and $\Pam_{N, c_1}\ll \Pam_{\rho_+}$.
In particular,
$\rho_-\leq \min_{k\in \T_N} E_{\mu^N_L}[\xi(k)] \leq \max_{k\in \T_N}E_{\mu^N_L}[\xi(k)]\leq \rho_+$.
\end{enumerate}
\end{Cond}

These are natural conditions and also those which allow to fit into the results in \cite{LPSX} and \cite{FHSX}, \cite{FX}:  (a) specifies an initial law of large numbers for the bulk particle numbers; we will chose later on in some of the results that $\rho_0\in C^\beta(\T)$ for different ranges of $\beta\geq 0$. (b) gives that the initial marginal may differ from the stationary measure $\Pam_N$ at $O(N)$ locations of $\T_N$ in the sense of relative entropy.  (c) specifies that the initial tagged particle position has a weak limit. (d) states that the full initial measure $\mu^N$ may differ from the stationary measure $\nu^N$ in the sense of $O(N)$-relative entropy; such a condition is useful to get control of later time distributions of the joint process.  Finally, (e) allows the monotone coupling mentioned earlier, so that we may bound and truncate local particle numbers; for instance, as in the next Section \ref{truncation section}.

 In the following, we will denote the process
measure and associated expectation governing $\eta_\cdot$ starting
from $\kappa$ by $\P_\kappa$ and $\E_\kappa$.  When the particle system process
starts from $\{\mu^N\}_{N\in \N}$
satisfying Condition
\ref{cond: initial measure}, 
we will denote by $\P_N:= \P_{\mu^N}$ and
$\E_N:= \E_{\mu^N}$, the associated process measure and expectation.

\subsection{Local equilibrium measures}
We observe that Condition \ref{cond: initial measure} is satisfied, for example, by
`local equilibrium' measures
$\left\{\mu^N_{\text{le}}\right\}_{N\in\N}$ associated to macroscopic
profiles $\rho_0\in C^\beta(\T)$ for $\beta\geq 0$ such that $0<\rho_-\leq \rho_0(\cdot)\leq \rho_+$ on $\T$.
  For each $N\in \N$, let
$\mu^N_{\text{le}}$ be the measure on
$\T_N\times \Sigma_N$ given by
\begin{equation}
\label{loc eq definition}
\mu^N_{\text{le}} (x, \xi)  \;=\;
\frac{\xi(x)}{\sum_{k\in \T_N} \rho_0(k/N)}\Pam_{N, \rho_0(\cdot)}(\xi),
\quad \text{for } k\in \T_N
\;. 
\end{equation}
Here, we define $\Pam_{N,\rho_0(\cdot)} =\Pam_N(\cdot; \{\tilde \phi_{k,N}\}) = \prod_{k\in \T_N}\Pb_{\tilde\phi_{k,N}}$ where the parameters $\{ \tilde\phi_{k,N}\}_{k\in \T_N}$ are such
that $E_{\Pam_{N, \rho_0(\cdot)}} [\xi(k)] = \rho_0(k/N)$.  We note when $\rho_0(\cdot)=\rho$ is constant, $\Pam_{N, \rho_0(\cdot)} = \Pam_\rho$.

Then, the marginals $\mu^N_{X, \text{le}}$ and $\mu^N_{L, \text{le}}$ are given by the following mass functions on $\T_N$ and $\Sigma_N$ respectively:
\begin{align*}
\mu^N_{X,\text{le}}(x) = \frac{\rho_0(x/N)}{\sum_{k\in \T_N}\rho_0(k/N)} \ \ {\rm and \ \ }
\mu^N_{L,\text{le}}(\xi) = \frac{\sum_{k\in \T_N}\xi(k)}{\sum_{k\in \T_N}\rho_0(k/N)} \Pam_{N, \rho_0(\cdot)}(\xi).
\end{align*}
One may verify that Condition \ref{cond: initial measure} (a) for the size biased measure $\mu_{L,\text{le}}^N$ holds straightforwardly with respect to macroscopic profile $\rho_0$ as $G$ is uniformly continuous by Chebychev and triangle inequalities.  Similarly, Condition \eqref{cond: initial measure} (c), as $\mu^N_{X,\text{le}}$ converges weakly to $\rho_0(x)/\|\rho_0\|_{L^1(\T)}dx$, holds also.

Also, Condition \ref{cond: initial measure} (b) holds since
$$\H(\mu^N_{L, \text{le}}| \Pam_N) = E_{\Pam_N}\Big[ \log \frac{\sum_k \xi(k)}{\sum_k \rho_0(k/N)} \Big] + \H(\Pam_{N,\rho_0(\cdot)}|\Pam_N).$$
By Lemma 3.4 in \cite{LPSX} or straightforward computation, we observe that $\H(\Pam_{N,\rho_0(\cdot)}|\Pam_N)\leq CN$.  Also, the first term is bounded by $E_{\Pam_N}\big[(\frac{1}{N}\sum_k \xi(k))/(\frac{1}{N}\sum_k \rho_0(k/N))]\leq C$ uniformly in $N$.

In addition, Condition \ref{cond: initial measure} (e) holds :  Note that the factors $\Pb_\phi$ are stochastically ordered in $\phi$.  Also, $\Phi(\rho_-)\leq \min_k \tilde \phi_{k,N} \leq \max_k \tilde \phi_{k,N}\leq \Phi( \rho_+)$.  Then, there are constants
$c_0,c_1$ such that $\max_{k\in \T_N} c_0\tilde \phi_{k,N} \leq \Phi(\rho_-) \leq \Phi(\rho_+)\leq \min_{k\in \T_N}c_1\tilde \phi_{k,N}$.  For increasing coordinatewise functions $f:\Sigma_N \rightarrow \R$, we have $E_{\Pam_{N,c_0}}[f(\xi)] \leq E_{\mu^N_{L, \text{le}}}[f(\xi)]\leq E_{\Pam_{N,c_1}}[f(\xi)]$.

The next lemma asserts that Condition
\ref{cond: initial measure} (d) holds to complete the verification.

\begin{proposition}
  There exists a constant $C(\rho_0) > 0$ depending only on $\rho_0$ such that 
$\H(\mu^N_{le} \,|\, \nu^N) \leq C(\rho_0) N$.
    \label{prop:rel_ent_order_N}
\end{proposition}

\begin{proof}
    We calculate
$\H(\mu_{le}^N \,|\, \nu^N) = E_{\mu_{le}^N} \left[ \log\left( \frac{\mu_{le}^N(X,\xi)}{\nu^N(X,\xi)} \right) \right]$.
   Write, canceling the term in common $\xi(x)$,
    \begin{equation*}
      \log  \frac{\mu_{le}^N(x, \xi)}{\nu^N(x, \xi)}
        =
      \log  \frac{\sum_{y \in \T_N} \rho_{y,N}}{\sum_{y \in \T_N} \rho_0(y/N)}  +\log \frac{\Pam_{N,\rho_0(\cdot)}(\xi)}{\Pam_N(\xi)}.
   \end{equation*}
   Then, 
   \begin{align}
   \label{help local eq}
   \H(\mu_{le}^N \,|\, \nu^N) = \log  \frac{\sum_{y \in \T_N} \rho_{x,N}(y)}{\sum_{y \in \T_N} \rho_0(y/N)} 
   + E_{\Pam_{N,\rho_0(\cdot)}}\left[\frac{\sum_{x\in \T_N}\xi(x)}{\sum_{x\in\T_N}\rho_0(x/N)}\log \frac{\Pam_{N,\rho_0(\cdot)}(\xi)}{\Pam_N(\xi)}\right].
   \end{align}
   Since the densities are bounded uniformly, the first term is $O(1)$.
    To treat the second term, write
     $$\log \frac{\Pam_{N,\rho_0(\cdot)}(\xi)}{\Pam_N(\xi)} = \sum_{x\in \T_N} \log \frac{\ZZ(\phi_{x,N})}{\ZZ(\tilde\phi_{x,N})} + \sum_{x\in \T_N}\xi(x)\log \frac{\tilde \phi_{x, N}}{\phi_{x,N}}.$$
     Since the fugacities are bounded, the deterministic term on the right-hand side of the above display is $O(N)$ and the other one is bounded by $C\sum_{x\in \T_N}\xi(x)$.  
     Then, the second term in \eqref{help local eq} is bounded by
     $O(N)E_{\Pam_{N, \rho_0(\cdot)}}\left[\frac{\sum_{x\in \T_N}\xi(x)}{\sum_{x\in \T_N}\rho_0(x/N)}\right] + CE_{\Pam_{N,\rho_0(\cdot)}}\left[\frac{\big(\sum_{x\in \T_N}\xi(x)\big)^2}{\sum_{x\in \T_N}\rho_0(x/N)}\right]$,
     which is $O(N)$ as desired.
     \end{proof}
 
\begin{remark}
\rm 
We observe other `local equilibrium measures' which locate the tagged particle at a fixed site $x_0=x_0^N\in \T_N$ may also be considered.  For instance, one may take $\mu^N_{\text{le}}(x_0, \xi) = \big(\OneB_{\{\xi(x_0)\geq 1\}}/E_{\Pam_{N, \rho_0(\cdot)}}[\OneB_{\{\xi(x_0)\geq 1\}}]\big)\Pam_{N, \rho_0(\cdot)}(\xi)$ and $\mu^N_{\text{le}}(x,\xi)=0$ for $x\neq x_0$, as opposed to \eqref{loc eq definition}, and verify that Condition \ref{cond: initial measure} is satisfied.  The sites $x_0=x_0^N$ may be chosen so that $x_0^N/N\rightarrow z\in \T$.
\end{remark}

\subsection{Local particle number estimates via attractiveness}
\label{truncation section}
While the total particle number $\sum_{x\in \T_N}\eta_t(x) = \sum_{x\in \T_N}\eta_0(x)$ for all $t\geq 0$ is conserved, and as the initial total number is $O(N)$ by Condition \ref{cond: initial measure} (e), such mass conservation does not provide control over local particle numbers $\eta_t(x)$ for a fixed $x\in \T_N$ that we will use in the sequel. 

\begin{lemma}
\label{loc trun lemma}
    Starting under an initial measure $\mu^N$ satisfying Condition \ref{cond: initial measure} (e),
     the probability that the maximum number of particles $\ds\max_{x\in\T_N} \eta_t(X^N_t+ x)$ at any time $t\geq 0$ is larger than $\log N$ is small: For all $C>0$, $\ds \limsup_{N \rightarrow \infty} \,\PB^N(\max_{x\in \T_N} \eta_t(X^N_t+x) \geq C \log N) = 0$.
\end{lemma}

\begin{proof}
We may rewrite, 
$
\max_{x\in \T_N} \eta_t(X^N_t+x) \ = \ \max_{x\in \T_N} \eta_t(x),
$
in terms of the system $\eta_t$, without the shift by $X^N_t$, which is attractive.
Since $\max_x \eta_\cdot(x)$ is an increasing function of $\eta_\cdot$, we have by attractiveness and Condition \ref{cond: initial measure} (e) that
\begin{align*}
\mathbb{P}_N \Big[ \max_x \eta_t(x)  \geq  C\log N \Big]
\le  \mathbb{P}_{\Pam_{N,c_1}}\Big[
\max_x \eta_t(x)\geq C\log N\Big]\;.
\end{align*}
Under the stationary measure $\Pam_{N,c_1}$, the variables $\{\eta_t(x)\}$
are independent with uniform over $x\in \T_N$ exponential
moments due to the (FEM) condition (cf. Section \ref{subsec: invariant measure}) satisfied by the marginals $\Pb_{\phi_{x,N}}$, and that the fugacities $\phi_{x,N}$ are uniformly bounded in $x\in \T_N$.  The claim follows as a consequence.
\end{proof}

%%%%%%%%%%%%%%%%%%%%%%%%%%%%%%%%%

\subsection{Hydrodynamics of the bulk density}
\label{sec03}
Our main theorems for the evolution of $X^N_t$ will involve the hydrodynamic bulk mass density of the system.   We now apply results in \cite{LPSX} to identify this density.
Consider the diffusively scaled particle mass empirical measure:
\begin{equation*}
\pi_t^N(dx)
:=
\dfrac 1 N 
\sum_{k=1}^N
\eta_t (k) \delta_{k/N}(dx)\;,
\end{equation*}
where $\delta_x$, stands for the Dirac mass at $x\in \T$.
Let $\MM_+$ be the space of finite nonnegative measures on $\T$, and
observe that $\pi^N_t\in \MM_+$.  We will place a metric
$d(\cdot,\cdot)$ on $\MM_+$ which realizes the dual topology of $C(\T)$
(see p. 49 of \cite{KL} for a definitive choice).  Here, the
trajectories $\{\pi^N_t: 0\leq t\leq T\}$ are elements of the
Skorokhod space $D([0,T],\MM_+ )$, endowed with the associated Skorokhod
topology.

In the following, for $G\in C(\T)$ and $\pi\in \MM_+$, denote
$\langle G, \pi\rangle = \int_\T G(u)\pi(du)$. 

We now recall the hydrodynamic limit (HDL) for the process $\eta_\cdot$ with respect to the
environment $\alpha_\cdot^N$ in \cite{LPSX}.

\begin{theorem}
\label{main thm}
For initial distributions $\mu_L^N$, where $\mu^N$ satisfies Condition \ref{cond:
  initial measure},
 for $t\in [0,T]$, test function
$G\in C^{\infty}(\T)$, and $\delta>0$, we have
\begin{equation*}
\lim_{N\to \infty} \P_{\mu^N_L}
\Big[ \Big|
\langle G, \pi^N_t \rangle
-
\int_\T G(x) \rho(t,x) dx
\Big|
>\delta
 \Big]
 =0,
\end{equation*}
where $\rho(t,x)$ is the unique weak solution of 
\begin{equation} \label{eqn: with ave}
\begin{cases}
\partial_t \rho(t,x)
=
\dfrac12 \partial_{xx} \Phi(\rho(t,x))
-
2 \partial_x \left(\alpha(x)
\Phi(\rho(t,x))\right),\\
\rho(0,x) = \rho_0(x),
\end{cases}
\end{equation}
in the class of `good' weak solutions given in Definition
\ref{weak_def} below.
\end{theorem}

\begin{definition}
\label{weak_def}
We say $\rho(t,x): [0,T]\times \T \mapsto
[0,\infty)$ is a good weak solution to \eqref{eqn: with ave}
if
\begin{enumerate}
\item[(1)] 
$\int_\T \rho(t,x) dx = \int_\T \rho_0(x) dx$ for all $t\in[0,T]$.
\item[(2)]
$\rho(t,\cdot)$ is weakly continuous, that is, for all $G\in C(\T)$, $\int_\T G(x) \rho(t,x) dx$ is a 
continuous function in $t$.
\item[(3)] There exists an $L^1([0,T]\times
\T)$ function denoted by $\partial_x
\Phi(\rho(s,x))$ such that, for all $G(s,x) \in C^{0,1} \left(
[0,T]\times \T \right)$, it holds
\begin{equation*} 
\int_0^T \int_\T 
\partial_x G(s,x) \Phi(\rho(s,x)) dx ds
=
-\int_0^T \int_\T 
G(s,x) \partial_x\Phi(\rho(s,x)) dx ds.
\end{equation*}
\item[(4)] 
For all $G(s,x) \in C_c^\infty \left( [0,T)\times \T  \right)$
\begin{align*}
&\int_0^T \int_\T \partial_sG(s,x) \rho(s,x) dx ds
+ \int_\T G(0,x) \rho_0(x) dx\\
&\qquad\qquad=
-\int_0^T \int_\T 
\left[
 \frac{1}{2}\partial_{xx}G(s,x) \Phi(\rho(s,x)) 
+
\partial_x G(s,x) [\, \alpha(x) \, \Phi(\rho(s,x)) \,]
\right]dx ds.
\end{align*}
\end{enumerate}
\end{definition}

\begin{remark}\rm
We comment  
that the proof of Proposition 9.2 in \cite{LPSX} yields, exactly in the way the proof of Theorem V.7.1 in \cite{KL} does in the translation-invariant setting when $\alpha^N_\cdot\equiv 0$, the additional bound $\int_0^Tds\int_\T du \frac{|\partial_x \Phi(\rho(s,u))|^2}{\Phi(\rho(s,u))}<\infty$, although we don't use it in this article.
\end{remark}

As a consequence of Theorem \ref{main thm}, we observe that both quenched and annealed hydrodynamic limits follow with respect to random environment 
$\left\{r_k^N := \frac{1}{\sigma N\varepsilon}\sum_{|j-k|\leq \lfloor N\varepsilon \rfloor} r_j \psi\left(\frac{j-k}{\varepsilon}\right): k\in \T_N \right\}$ (cf. Section \ref{quenched-section}).

Let 
$E_P\P^{\omega}_N$ be the annealed probability measure, where $P(d\omega)$ governs the random environment $\omega = \{r_x\}_{x\in \N}$ and $\P^{\omega}_N = \P_N$ is the process measure of the speeded up zero-range process $\eta$, with single particle jump rates $(1/2) \pm r^N_k/\sqrt{ N}$ to the left and right from location $k\in \T_N$, conditioned on the environment.  Recall that $\{q_k^N\}\stackrel{d}{=} \{r_k^N\}$ and that a.s. $\sqrt{N}q_k^N \rightarrow W'_\varepsilon(x)$ when $k/N\rightarrow x$ (cf. \eqref{eqn: q and W}, \eqref{q_conv}).

\begin{theorem}
\label{main_Cor} With respect to the random environments, we have the hydrodynamic limits:

(I) (Quenched HDL) For almost all realizations $\omega$, the statement
of Theorem \ref{main thm} holds with respect to
$\alpha^N_\cdot = \sqrt{N}q_\cdot^N$ and $\alpha = W'_\varepsilon$.

(II) (Annealed HDL) Moreover, the law of $\pi_{\cdot}^N$, under
$E_P\P^\omega_N$, converges weakly to the law of
$\rho(\cdot,x)dx= \rho(\cdot, x; \Phi, W'_\varepsilon)dx$ with respect
to the distribution of Brownian motion $W$.
\end{theorem}

\subsection{Boundedness and regularity of the good weak solution $\rho$}
\label{bounded section}

It will be useful to specify an a priori bound of $\rho(t,x)$, the good weak solution of \eqref{eqn: with ave}, as well as its regularity depending on the smoothness of $\alpha$.
Define, 
for $\ell\geq 1$,
\begin{equation}
\label{random density}
\eta^\ell(x) := \frac{1}{2\ell +1}\sum_{|k-x|\leq \ell}\eta(k).
\end{equation}
Then, $\eta_t^{\theta N}(\lfloor uN\rfloor) 
				= \frac{2\theta N}{2{\color{blue}\theta} N + 1} \big\langle \frac{1}{2\theta}\OneB_{[-\theta +\tfrac{1}{N}\lfloor uN\rfloor, \theta + \tfrac{1}{N}\lfloor uN \rfloor]},~ \pi_t^N \big\rangle$, for $\theta>0$.

\begin{lemma}
\label{lem:smoothness}
    With initial measure $\mu^N$ satisfying Condition \ref{cond: initial measure}, the following holds.
		
		1. 
		There exists positive constants $\rho_-, \rho^+$ such that the solution to \eqref{eqn: with ave} satisfies $\rho_- \leq \rho(t,x) \leq \rho^+$ for a.e $(t,x) \in [0,T]\times\T$.  
	As a consequence, $b_0\leq \Phi'(\rho(t,x)), \frac{\Phi(\rho(t,x))}{\rho(t,x)}\leq b_1$ for $(t,x)\in [0,T]\times \T$, for $0<b_0\leq b_1<\infty$.

    2. (a) Suppose that $\alpha\in C(\T)$ and $\rho_0\in C^\beta(\T)$ for $\beta>0$.
		Then, we have $\rho(t,x)\in C^{\gamma/2, \gamma}([0,T]\times \T)$ for 
		a $0<\gamma\leq\beta$.  
		
		(b) Now suppose that $\alpha\in C^\beta(\T)$ and $\rho_0\in C^{1+\beta}(\T)$ for 
		$\beta>0$.    
		Let $u(t,x) = \int_0^x \rho(t, z)dz$ on $[0,T]\times \T$. Then, $u(t,x)\in C^{(2+\gamma)/2, 2+\gamma}([0,T]\times \T)$, for a $0<\gamma\leq \beta$, is a classical solution of  
		\begin{align}
		\label{u-equation}
    \partial_t u(t,x) = \frac{1}{2} \Phi'(\nabla u) \Delta u - 2\alpha(x) \Phi(\nabla u),
\end{align}
and as a consequence the unique para-controlled solution of \eqref{u-equation}, certainly belonging to space $\LL_T^{3/2-}=C([0,T], \hat{C}^{3/2-}(\T))\cap C^{(3/2-)/2}([0,T], L^\infty(\T))$.

 Consequently, $\rho(t,x)=\nabla u(t,x)\in C^{(2+\gamma)/2, 1+\gamma}([0,T]\times \T)$.
\end{lemma}

Here, 
$\hat{C}^\beta(\T):=B^\beta_{\infty,\infty}(\T, \R)$ is the spatial H\"older-Besov space equipped with the norm $\|\cdot\|_{\hat{C}^\beta}$; see page 62, Appendix A.1 in \cite{GIP} for the exact definition.  When $\beta\in (0,\infty)\setminus\N$, it is known that $\hat{C}^\beta(\T)=C^\beta(\T)$, the standard H\"older space used in \cite{lady} and \cite{Lieberman}, and also $C^\beta(\T) \subset \hat{C}^\beta(\T)$ when $\beta\in \N$; see page 62 in \cite{GIP} and page 99 in \cite{BCD11}.

Then, $C_T\hat{C}^\beta :=C([0,T], \hat{C}^{\beta}(\T))$ is the space of $\hat{C}^\beta$-valued continuous functions on $[0,T]$ with the supremum norm $\|f\|_{C_T\hat{C}^\beta}=\sup_{0\leq t\leq T}\|f(t)\|_{\hat{C}^\beta}$, and $C^{\beta/2}_TL^\infty:=C^{\beta/2}([0,T], L^\infty(\T))$ is the space of $\beta/2$-H\"older continuous functions from $[0,T]$ to $L^\infty(\T)$ with the semi-norm $\|f\|_{C^{\beta/2}_TL^\infty} = \sup_{0\leq s\neq t\leq T} \|f(t)-f(s)\|_{L^\infty(\T)}/|t-s|^{\beta/2}$.  The norm on $\mathcal{L}^\beta_T$ is $\|f\|_{\mathcal{L}^\beta_T} = \|f\|_{C_T\hat{C}^\beta} + \|f\|_{C^{\beta/2}_TL^\infty}$.

\begin{proof}[Proof of Lemma \ref{lem:smoothness}]		
		Let $S^N_t$ be the semigroup associated to $L$, generating the standard process \eqref{eqn: generator L}.  By assumption (A) and Condition \ref{cond: initial measure} (e), 
		for the attractive system, we have
    the stochastic ordering
		\begin{equation}
        \Pam_{\rho_-}\ll \Pam_{N, c_0}\leq \mu_L^N S^N_t 
        \leq \Pam_{N, c_1}\ll \Pam_{\rho_+}.
    \end{equation}
 Then, $E_{\Pam_{\rho_-}}[f] \leq
    E_{\mu_L^NS^N_t}[f] \leq
    E_{\Pam_{\rho_+}}[f]$ for any monotone function $f$ on $\N_0^{\T_N}$.

		Set $f(\eta) = \eta^{\theta N}(\lfloor Nu\rfloor)$.  The hydrodynamic limit (Theorem \ref{main thm}) gives that $\rho_- \leq \frac{1}{2\theta}\int_{-\theta}^\theta \rho(t, u+v)dv\leq \rho_+$.  Hence, $\rho_-\leq \rho(t, u)\leq \rho_+$ a.e. with respect to Lebesgue measure on $\R$ as desired.
		The bounds on $\Phi'(\rho)=\Phi(\rho)/\sigma^2(\rho)$ (cf. after \eqref{Phi_eqn}) and $\Phi(\rho)/\rho$ follow as $\Phi$ is strictly increasing, and $\Phi(\rho)>0$ and $\sigma^2(\rho)>0$ for $\rho>0$.

One may deduce Part 2 (a), as $\rho(t,x)$ is bounded by Part 1, and as $\alpha(x)$ and $\Phi(\rho)$ are bounded by assumption and Part 1 again, by the regularity Theorem VI.6.33 in \cite{Lieberman} for weak solutions. 

For Part 2 (b), that $u$ is a classical solution now follows via Theorem V.5.14 in \cite{Lieberman} as $\Phi'(\rho), \alpha\Phi(\rho)$, thought of as functions of $(t,x)$, are bounded and belong to $C^{\gamma/2, \gamma}([0,T]\times \T)$ for a $0<\gamma\leq \beta$ by Part 1 and Part 2 (a). [We note in passing that Theorems 6.1 and 6.4 in \cite{lady} are used in Section 2.1, pages 866-867 in \cite{FX} to deduce regularity of $\rho(t,x)$ under stronger assumptions on $\alpha (=\xi\in C^2(\T) \text{ in the notation there})$ and $\rho_0\in C^{2+\beta}(\T)$ for $\beta\in (0,1)$.]

Moreover, as $u$ is a classical solution of \eqref{u-equation}, we claim it is the unique para-controlled solution in $\LL_T^{3/2-}$.  Indeed, the assumptions on $\xi = \alpha \in C^\beta(\T)$
(in the notation of \cite{FHSX},  \cite{FX}) and $\rho_0\in C^{1+\beta}(\T)$, since $\beta>0>-1/2-$ and $1+\beta>1/2-$, satisfy the conditions in \cite{FHSX}, \cite{FX} for $u$ to be in $\LL_T^{\gamma}$ for all $\gamma\in (13/9, 3/2)$.  Since $u$ is classical, all terms in \eqref{u-equation} are well-defined as continuous functions and \eqref{u-equation} coincides with the fixed-point problem definition of the unique local-in-time para-controlled solution in Section 2 in \cite{FHSX}, extended to global-in-time in \cite{FX}.
\end{proof}

\begin{remark}\rm
We note the question of whether $u(t,x)$, corresponding to a good weak solution $\rho(t,x)$ without further assumptions, is
the unique para-controlled solution of \eqref{u-equation} in a space $\LL_T^\gamma$ is of interest and left for future investigation.
Part of the reason for the choice in Part 2(b) of Lemma \ref{lem:smoothness}, and later in Theorems \ref{singular hyd thm} and \ref{thm singular sde} to choose $\rho_0\in C^{1+\beta}(\T)$ and $\alpha\in C^\beta(\T)$ for $\beta>0$, or $\alpha = W'_\varepsilon\in C^{1/2-}(\T)$ is that $u(t,x)$ is guaranteed to be the para-controlled solution of \eqref{u-equation}. 
\end{remark}

\subsection{Singular hydrodynamic limit as $\varepsilon\downarrow 0$}
When $\alpha = W'_\varepsilon$, we now discuss the limit of $\rho=\rho^\varepsilon$ as $\varepsilon\downarrow 0$, which follows from \cite{FHSX}, \cite{FX}.

For the following result, we will specify enough smoothness of $\rho_0$ (the same initial condition for all $\rho^\varepsilon$) to fit into the framework of Part 2(b) of Lemma \ref{lem:smoothness}.  
Recall also 
that $\Phi$ is $C^\infty([0,\infty))$ and by (LG), (M) that $g_*\leq \Phi(\rho)/\rho\leq g^*$, which allow us to fit into the framework of \cite{FHSX}, \cite{FX}.

\begin{theorem}
\label{singular hyd thm}
 Fix $\rho_0\in C^{1+\beta}(\T)$ for 
$\beta>0$.
With respect to $\varepsilon>0$ and $\alpha = W'_\varepsilon\in C^{1/2 - }(\T)$ as in \eqref{alpha defn}, the solutions $\rho^\varepsilon$ of \eqref{eqn: with ave} with initial condition $\rho_0$ converge in probability, with respect to the probability measure $\PB_W$ governing $W$, to the solution $\rho^0$ of the SPDE
\begin{align*}
    \partial_t \rho^0 = \frac{1}{2}\Delta \left(\Phi(\rho^0)\right) - 2\nabla\left(W'\Phi(\rho^0)\right)
\end{align*}
in the space $\mathcal{L}^{1/2-}_T$, 
which is the gradient $\rho^0=\nabla u^0$ with respect to the unique para-controlled solution of $\partial_t u^0=\frac{1}{2}\Phi'(\nabla u^0)\Delta u^0 - W'\Phi(\nabla u^0)$.

In particular, $\rho^\varepsilon(t, x)$ converges uniformly to $\rho^0(t, x)$ for $(t,x)\in [0, T]\times \mathbb{T}$, in probability with respect to $\PB_W$.
\end{theorem}

\begin{proof} We give the relevant citations.
By Part 2(b) of Lemma \ref{lem:smoothness}, $\rho^\varepsilon$ is the gradient of the unique para-controlled solution of \eqref{u-equation}.  Then, Theorem \ref{singular hyd thm} follows by Theorem 1.1, Theorem 1.3 and Remark 1.2 in \cite{FX}, noting the noise $W'_\varepsilon \in C^{2-\gamma}(\T)$ for $\gamma \in (13/9, 3/2)$ and $\rho_0\in C^{1+\beta}(\T)$ and $1+\beta> 1/2$ thereby satisfying the assumptions in \cite{FX}, which shows existence/uniqueness of the para-controlled solution $\rho^0$ global-in-time, extending the local-in-time convergence shown in probability, Theorem 1.1 and Lemma 5.2 in \cite{FHSX}, to all times.
\end{proof}

\section{Results}
\label{ch:main_thm}

To state results for a tagged particle, we recall the interpretation of a diffusion $x_t$ on a torus as one `unwrapped' on $\R$ with periodic coefficients.
Let $\hat x_t=\hat x^z_t$ be an It\^{o} diffusion on $\R$
satisfying with respect to a probability space with admissible filtration the SDE
\begin{equation}
    d\hat x_t = b(t, \hat x_t) \,dt + \sigma(t, \hat x_t) \,dB_t ~~,\quad\quad \hat x_0 =z,
    \label{eq:diff_torus_sde_R}
\end{equation}
for $z\in \R$.
When the functions $b(t, \cdot)$ and $\sigma(t, \cdot)$ are periodic with period $1$ for all $t \in [0, T]$, we 
may write
\begin{align*}
\hat x_t -z= \int_0^t b(s, \hat x_s - \lfloor \hat x_s\rfloor)ds + \int_0^t \sigma(s, \hat x_s - \lfloor \hat x_s\rfloor)dB_s.
\end{align*}

We may understand the diffusion $x_t=x^Z_t$ on the torus $\T=(0,1]$ (with the usual distance $d(a,b)=\min\{|a-b|, 1-|a-b|\}$), satisfying
$$dx_t = b(t, x_t)dt + \sigma(t, x_t)dB_t, \quad\quad x_0\stackrel{d}{=} Z$$  
for $t\geq 0$, where $Z$ is a random variable, via the mixture relation
$x_t = \hat x^Z_t - \lfloor \hat x^Z_t\rfloor\in \T$, with $z=Z$, whose process measure is $E_Z[\hat{\mathcal{P}}^z]=\int \hat{\mathcal{P}}^zdP(Z\in dz)$, where $\hat{\mathcal{P}}^z$ is the process measure of $\hat x^z_t$.

 If there is a unique weak solution $\hat x^z$ for each $z\in \R$, 
 then the process measure of $\hat x^Z_t$ is uniquely given as $E_Z[\hat{\mathcal{P}}^z]$
 (cf. Proposition 1 and remark following in \cite{Kallenberg}).  Correspondingly, in this situation,
 \begin{equation}
 \label{process measure}
 \text{ the process measure of}\ x^Z_t\ \text{is} \ E_Z[\mathcal{P}^z]
 \end{equation}
  where $\mathcal{P}^z$ is the process measure of $\hat x^z_t - \lfloor \hat x^z_t\rfloor$ for a fixed $z$.

Our first main result identifies a diffusion limit for $\frac{1}{N}X^N_t=\frac{1}{N}X_{N^2t}$ with respect to an $\alpha_\cdot^N$ environment, when the initial condition has some smoothness to guarantee smoothness of $\rho(t,x)$ (cf. Lemma \ref{lem:smoothness}).

\begin{theorem}
\label{thm alpha tg}
    Suppose $\alpha\in C(\T)$ and $\rho_0\in C^\beta(\T)$ for $\beta>0$.  Let the initial measures $\mu^N$ of the process $(X^N_t, \eta_t)$ satisfy Condition \ref{cond: initial measure}.
		Then, $\frac{1}{N} X^N_{t}$ for $t\in [0,T]$ converges in distribution in the uniform topology to a diffusion $x_t$ on $\T$ given by the unique weak solution of the SDE
    \begin{equation}
        dx_t = 2\frac{\Phi(\rho(t, x_t))}{\rho(t, x_t)} \,\alpha(x_t) \,dt
        +
        \sqrt{ \frac{\Phi(\rho(t, x_t))}{\rho(t, x_t)} } \,dB_t ~~~,~~~ x_0 \stackrel{d}{=} Z_X,
        \label{eq:thm:sde:alpha}
    \end{equation}
		whose process measure is given by \eqref{process measure}, 
    where $B$ is a standard Brownian motion, $\rho(t,u)$ is the hydrodynamic density specified in \eqref{eqn: with ave}, and $Z_X$ is the limit in distribution of $\frac{1}{N}X^N_0$.
 \end{theorem}

\begin{remark}
\label{more smoothness rmk}
\rm
We comment, with more smoothness, say $\alpha(\cdot)\in C^{\beta'}(\T)$ for $\beta'\geq 1$ (and therefore Lipschitz) and $\rho_0\in C^{1+\beta''}(\T)$ for $\beta''>0$,
we have $\rho(t,u)\in C^{(2+\gamma)/2, 1+\gamma}([0,T]\times \T)$ for a $0<\gamma\leq \min\{\beta', \beta''\}$ by Part 2 (b) of Lemma \ref{lem:smoothness}.
Then, the SDE above would have bounded, Lipschitz coefficients, and $x_t$ would be the unique pathwise solution of \eqref{eq:thm:sde:alpha}, implying strong existence by the Yamada-Watanabe theorem (cf. Theorem IX.1.7 in \cite{revuz2004continuous}).
\end{remark}

Next, as a consequence, we state quenched and annealed results.  Recall the quenched random environment formulation in Section \ref{quenched-section}.

\begin{corollary}
\label{cor-tg-quenched}
Consider the seting of Theorem \ref{thm alpha tg} where $\rho_0\in C^\beta(\T)$ for 
$\beta>0$ and now $\alpha = W'_\varepsilon\in C^{1/2 - }(\T)$ as in \eqref{alpha defn} with respect to $\varepsilon>0$.  Then, under a quenched random environment $\omega$, we have $\frac{1}{N}X^N_{t}$ for $t\in [0,T]$ converges in distribution in the uniform topology to the unique weak solution $x^\varepsilon_t=x_t$ on $\T$ satisfying
\begin{equation}
        dx_t = 2\frac{\Phi(\rho(t, x_t))}{\rho(t, x_t)} \,W'_\varepsilon(x_t) \,dt
        +
        \sqrt{ \frac{\Phi(\rho(t, x_t))}{\rho(t, x_t)} } \,dB_t, \quad\quad x_0\sim Z_X,
       \label{eq:thm:sde}
    \end{equation}
    where $\rho=\rho^\varepsilon$ is the hydrodynamic density specified in Theorem \ref{main_Cor} with initial condition $\rho_0$.
    
    Moreover, under the annealed measure $\E_{\PB_W}\P^N$, the law of $\frac{1}{N}X^N_t$ converges weakly to the law of $x^\varepsilon_t$ under $\PB_W$, the law of $W$.
\end{corollary}

 \begin{proof}
The quenched part follows from Theorem \ref{thm alpha tg}.  The annealed result follows straightforwardly from the quenched one:  If $\E^N[F(\tfrac{1}{N}X^N_\cdot)] \rightarrow E[F(x^\varepsilon_\cdot)]$ for bounded, continuous $F:D([0,1], \T)\rightarrow \R$, then $\E_{\PB_W}\E^N[F(\tfrac{1}{N}X^N_\cdot)]\rightarrow \E_{\PB_W}E[F(x^\varepsilon_\cdot)]$, by bounded convergence.
\end{proof}

We now consider the limit of the diffusions $x^\varepsilon_t$ as $\varepsilon\downarrow 0$.
Consider the diffusion $\hat x^0_t = \hat x^{0,z}_t$ on $\R$, formally given with periodic drift and diffusion coefficient by
 \begin{equation*}
        d\hat x^0_t = 2\frac{\Phi(\rho^0(t, \hat x^0_t))}{\rho^0(t, \hat x^0_t)} \,W'(\hat x^0_t) \,dt
        +
        \sqrt{ \frac{\Phi(\rho^0(t, \hat x^0_t))}{\rho^0(t, \hat x^0_t)} } \,dB^0_t, \quad\quad \hat x^0_0 =z
    \end{equation*}
    where $z\in \R$.
More carefully, as with Brox diffusion, we specify it in terms of scale:  Namely, for $t\geq 0$, define
		$\hat x_t^0 = s_0^{-1}\big(s_0(z)+B^0(T_0^{-1}(t))\big)$
		 on $\R$ where 
		\begin{align}
		     \label{prop:s0_lim1}
        s_0(x) &:= \int_0^x \exp \Big(-4 \left\{ \lfloor y \rfloor W(1) + W(y - \lfloor y \rfloor) \right\} \Big)\,dy\\
    T_0(t) &:= \int_0^t \exp\Big( 8 \big\{ \lfloor s_0^{-1}(s_0(z)+B^0_r) \rfloor W(1) + W\big(s_0^{-1}(s_0(z)+B^0_r) - \lfloor s_0^{-1}(s_0(z)+B^0_r) \rfloor \big) \big\} \Big) \nonumber\\
    &\quad\quad\quad\quad\quad\quad\quad\quad \times \frac{1}{\chi_0(t, s_0^{-1}(s_0(z)+B^0_r))} \,dr.\nonumber
\end{align}
Here,
$B^0$ is a standard Brownian motion and
$\chi_0(t,x) = \Phi(\rho^0(t,x))/\rho^0(t,x)$.

We now impose a bit more smoothness on $\rho_0$ to guarantee that $u(t,x)$ solving \eqref{u-equation} with $\alpha = W'_\varepsilon$ is a classical solution and therefore a para-controlled solution, so as to access the para-controlled limits in \cite{FHSX}, \cite{FX}.

\begin{theorem}
\label{thm singular sde}
Consider the setting of Corollary \ref{cor-tg-quenched}.  Let now $\rho_0\in C^{1+\beta}(\T)$ where $\beta>0$.
Let also $x^0_0 \stackrel{d}{=} Z_X$.
 Define the process $x^0_t$ for $t\in [0,T]$ on $\T$ as a mixture $x^0_t = \hat x^{0,Z_X}_t - \lfloor \hat x^{0,Z_X}_t\rfloor$, where $z=Z_X$.
    We have that the diffusions $x_t^\varepsilon$ on $\T$, specified in \eqref{eq:thm:sde}, converge in distribution to $x_t^0$ as $\varepsilon \rightarrow 0$ in $\PB_W$-probability.

 Moreover, as a consequence, the laws of $x_t^\varepsilon$ under the annealed measure $\PB_W$ converge in distribution to the law of $x^0_t$ under $\PB_W$ as $\varepsilon\rightarrow 0$.
\end{theorem}

We will prove the quenched part of Theorem \ref{thm singular sde} in 
Section \ref{sec:varep-to-0}. The annealed statement follows as a consequence as in the proof of Corollary \ref{cor-tg-quenched}.

\begin{remark}\label{bridge rmk}
\rm
We comment,
with respect to a Brownian-bridge environment (see after \eqref{alpha defn}), where the disorder $a^N_k = \alpha^N_k - \frac{1}{N}\sum_{j\in \T_N}\alpha^N_j$ and $a(u)=W'_\varepsilon(u)-\int_\T W'_\varepsilon(v)dv$ has the same smoothness as $\alpha(u)=W'_\varepsilon$, one may deduce the corresponding version of Corollary \ref{cor-tg-quenched} from Theorem \ref{thm alpha tg}, with $a(\cdot)$ in place of $\alpha(\cdot)$.  

Also, the corresponding version of singular hydrodynamic limit Theorem \ref{singular hyd thm}, in terms of para-controlled solution $\rho^{a,0}(t,x)$, holds with respect to $a(\cdot)$ and formal limit $W'(\cdot) - W(1)$ in place of $\alpha(\cdot)$ and $W'$.  Then, a corresponding singular diffusion limit Theorem \ref{thm singular sde} holds with
$s_{a,0}$, $T_{a,0}$ and $\rho^{a,0}$ in place of $s_0$, $T_0$ and $\rho^0$ where
\begin{align*}
s_{a,0}(x)&:= \int_0^x \exp\Big (-4W(y-\lfloor y\rfloor) + 4W(1)(y-\lfloor y\rfloor)\Big)dy\\
T_{a,0}(t)&:= \int_0^t \exp\Big(8W(s_{a,0}^{-1}(s_{a,0}(z)+ B^0_r) - \lfloor s_{a,0}^{-1}(s_{a,0}(z)+ B^0_r)\rfloor)\\
&\quad\quad\quad -8 W(1)(s_{a,0}^{-1}(s_{a,0}(z)+ B^0_r) - \lfloor s_{a,0}^{-1}(s_{a,0}(z)+ B^0_r)\rfloor)\Big)\frac{1}{\chi_0(t, s_{a,0}^{-1}(s_{a,0}(z)+ B^0_r))}dr.
\end{align*}
\end{remark}

%%%%%%%%%%%%%%%%%%%%%%%%%%%%%%

\section{Proof of Theorem \ref{thm alpha tg}: Diffusion limit in an $\alpha(\cdot)$ environment}
\label{sec:varep>0}

We develop a martingale representation for $X^N_t$, discuss homogenization via a Replacement Lemma \ref{replacement-lemma}, and associated tightness of terms in the representation in the following subsections, before proving Theorem \ref{thm alpha tg} in Section \ref{identification subsection}.

\subsection{Martingale representation}
Let $J_t^{-}$ (respectively $J_t^{+}$) be the total number of jumps up to time $t$ by the tagged particle $X^N_\cdot$ to the left (respectively right). Such a process is a function of $(X^N_\cdot, \eta_\cdot)$. We may express the tagged particle location as $X^N_0 + J_t^{+} - J_t^{-}$ modulo $N$. 
These counting processes $J_t^\pm$ are compensated in terms of the jump rates $N^2\frac{g(\eta_t(X^N_t))}{\eta(X_t)}p^{N, \pm}_{X^N_t}$ so that 
\begin{equation*}
    M_t^{\pm} =
    J_t^{\pm} - N^2\int_{0}^t \frac{g(\eta_s(X^N_s))}{\eta_s(X^N_s)} \,p_{X^N_s}^{N, \pm} \,ds
\end{equation*}
are martingales.  One may also compute the quadratic variations, noting that the size of the jumps are $1=|\pm 1|$, as
\begin{equation*}
       \langle M^{\pm} \rangle_t 
        = N^2\int_0^t \frac{g(\eta_s(X^N_s))}{\eta_s(X^N_s)} \, p_{X^N_s}^{N,\pm} \,\,ds.
\end{equation*}
Since the jumps are not simultaneous, the martingales $M^+_t$ and $M^-_t$ are orthogonal, and therefore their cross variation vanishes.

Since $p_{X^N_s}^{N,+} - p_{X^N_s}^{N,-}=\frac{2}{N}\alpha^N_{X^N_s}$, we obtain a representation for $X^N_t$:
\begin{equation}
\frac{1}{N}X_t^N-\frac{1}{N}X^N_0 = 2\int_{0}^t \frac{g(\eta_s^N(X_s^N))}{\eta_s^N(X_s^N)} \,\alpha^N_{X^N_s}\,ds +  \frac{1}{N}M^N_t,
    \label{eq:Xt}
\end{equation}
modulo $1$ (actually supported on $\frac{1}{N}\T_N$), where $ M^N_t=M_{t}^{+} - M_{t}^{-}$.  Alternatively, one can extend $X^N_\cdot$ periodically to $\Z$, in which case \eqref{eq:Xt} would be an equation on $\R$ say.

The quadratic variation of $\frac{1}{N}M^N_t$,
adding the quadratic variations of $\frac{1}{N}M^{\pm}$, 
equals
\begin{equation}
    \frac{1}{N^2}\langle M^{N} \rangle_t 
    = \int_0^t \frac{g(\eta^N_s(X^N_s))}{\eta^N_s(X^N_s)} \,ds.
    \label{eq:qv}
\end{equation}

\subsection{Homogenization of rates}
To take limits in the martingale representation \eqref{eq:Xt}, we would like to replace the local rate $g(\eta_t(X_t^N))/\eta_t(X_t^N)$ by its appropriate continuum homogenization.
Since $\frac{g(n)}{n}$ is bounded and Lipschitz by our assumptions on $g$ (Section \ref{g assump section}),
 we state a `replacement' Lemma \ref{replacement-lemma} with respect to a bounded, Lipschitz function $h$ evaluated at coordinate $\eta_\cdot(X^N_\cdot)$.

 As mentioned in the introduction, we will replace $h(\eta_t(X^N_t))$ by its average with respect to a localized `equilibrium' distribution.  We will be able to show effects of the random environment are minimal in this localization.
At time $N^2t$, the system has settled so that near the location $X^N_t$ one expects the local configuration $\eta_t(X^N_t + x)$ for $1\leq x\leq \epsilon N$ to be distributed by a stationary distribution $\Pam_\cdot$ for the translation-invariant system indexed by the local random density $\eta_t^{\theta N}(X^N_t+x)$ (cf. \eqref{random density}).

This is formulated as the following result.

\begin{lemma}[The Replacement Lemma]
\label{replacement-lemma}
    Let $h$ be a bounded and Lipschitz function.  Let also $D^N_{x}$ be a bounded function of $x\in \T_N$, uniformly for all $N\in \N$. Then, for $t\in [0,T]$,
    \begin{equation}
        \limsup_{\ell \rightarrow \infty} \limsup_{\epsilon \rightarrow 0} \limsup_{\theta\rightarrow 0} \limsup_{N \rightarrow \infty} ~\E^N\left[ \left| \int_0^t D^N_{X^N_s}\Big(h(\eta_s(X^N_s)) - \frac{1}{\epsilon N} \sum_{x = 1}^{\epsilon N} \overline{H_{\ell}}(\eta_s^{\theta N}(X^N_s + x)) \Big)\,ds \right| \right] = 0
    \end{equation}
    where
    \begin{enumerate}
        \item $H(\rho) = E_{\nu_\rho^0}[h(\xi(0))]$ with $\ds \nu_\rho^0 = \frac{\xi(0)}{\rho} \Pam_\rho(\xi)$,
        \item $\overline{H_{\ell}}(\rho) = E_{\Pam_\rho}[H(\eta^\ell(0))]$ and $\Pam_\rho$ is defined in Section \ref{stat measures section}.
    \end{enumerate}
\end{lemma}

\begin{proof} The proof is divided into several steps. Each step will be dealt with in its own later section. These steps are: local 1-block in Section \ref{sec:loc_1b}, local 2-block in Section \ref{sec:loc_2b}, and global replacement in Section \ref{sec:glob_repl}.
   By the triangle inequality, we have
    \begin{align}
    \label{eq:lem:repl:split}
        &\E^N\left[ \left| \int_0^t D^N_{X^N_s}\Big(h(\eta_s(X^N_s)) - \frac{1}{\epsilon N} \sum_{x = 1}^{\epsilon N} \overline{H_\ell}(\eta_s^{\theta N}(X^N_s + x)) \Big)~ds \right| \right]\\
          &\ \ \ \ \   \leq
            \E^N\left[ \left| \int_0^t D^N_{X^N_s}\Big(h(\eta_s(X^N_s)) -  H(\eta_s^{\ell}(X^N_s))\Big) ~ds \right| \right]\nonumber\\
            &\ \ \ \ +
            \E^N\left[ \left| \int_0^t D^N_{X^N_s}\Big(H(\eta_s^{\ell}(X^N_s)) - \frac{1}{\epsilon N} \sum_{x = 1}^{\epsilon N} H(\eta_s^{\ell}(X^N_s + x))\Big) ~ds \right| \right]\nonumber\\
           &\ \ \ \ +
            \E^N\left[ \left| \int_0^t D^N_{X^N_s}\Big(\frac{1}{\epsilon N} \sum_{x = 1}^{\epsilon N} H(\eta_s^{\ell}(X^N_s + x)) -   \frac{1}{\epsilon N}\sum_{x=1}^{\epsilon N}\overline{H_\ell}(\eta_s^{\theta N}(X^N_s + x))\Big) ~ds \right| \right].\nonumber
    \end{align}
    
    As $h$ is bounded, Lipschitz, it is known that $H$ is bounded, Lipschitz; see Lemma 6.7 in \cite{JLS}.
    Then, each term on the right hand side of \eqref{eq:lem:repl:split} vanishes in the limit by Lemmas \ref{lem:L1b}, \ref{lem:L2b}, and \ref{lem:G} respectively.
\end{proof}

\subsection{Tightness of constituent processes}
We now state that the processes
$$\XX_N = (\pi_{\cdot}^N, \tau_{X_{\cdot}^N}\pi_{\cdot}^N, \tfrac{1}{N}X_{\cdot}^N, \tfrac{1}{N}M_{\cdot}^N, \langle \tfrac{1}{N}M^N \rangle_{\cdot} )$$ are tight in the associated Skorohod space $\DD([0,T]) = \DD([0, T],~ \MM_+(\T) \times \MM_+(\T) \times \T \times \R \times \R)$ endowed with the Skorohod topology.  In fact, we will show that $\XX_N$ is tight with respect to the uniform topology, associated with
$\mathcal{C}([0,T]) =  \mathcal{C}([0, T],~ \MM_+(\T) \times \MM_+(\T) \times \T \times \R \times \R)$. 
Here, $\tau_x$ is the shift operator so that
$$\tau_x\pi^N_\cdot = \frac{1}{N}\sum_{z\in \T_N}\eta_\cdot(z+x)\delta_{z/N}.$$

\begin{theorem}
\label{tightness thm}
The sequence $Q^N$ of distributions of $\XX_N$, belonging to $\DD([0,T])$, is tight with respect to the uniform topology of $\mathcal{C}([0,T])$.
\end{theorem}

\begin{proof}
Let $A^N_t = \int_0^t \frac{g(\eta_s(X_s^N)}{\eta_s(X^N_s)}ds$ and $C^N_t = \int_0^t \frac{g(\eta_s(X_s^N)}{\eta_s(X^N_s)} \alpha^N_{X^N_s}ds$ so that $\tfrac{1}{N}(X^N_t - X^N_0) = \tfrac{1}{N}M^N_t + C^N_t$, modulo $1$, in \eqref{eq:Xt}.  Since $g(n)\leq g^*n$ by assumption (LG), and $\sup_N\sup_{x\in \T_N}|\alpha^N_{x}|<\infty$, we have the processes $A^N_\cdot$ and $C^N_\cdot$ are tight with respect to the uniform topology by Kolmogorov-Centsov criterion.  Since $\langle \tfrac{1}{N}M^N\rangle_\cdot = A^N_\cdot$ (cf. \eqref{eq:qv}), we have also that $\langle \tfrac{1}{N}M^N\rangle_\cdot$ is tight.

Since $\tfrac{1}{N}X^N_0$ is assumed to converge weakly to $Z_X$, the initial value is therefore tight.  
The same argument as in Lemma 3.3 in \cite{JLS} shows that $\tfrac{1}{N}X^N_\cdot$ is tight in the uniform topology.  Indeed, the jump rate of $X^N$ is $g(\eta(X^N))/\eta(X^N)\leq g^*$, and so one can couple $X^N_\cdot$ with $Z^N_\cdot$ a random walk with rate $g^*$ on $\T_N$ with the same skeleton but holding times less than or equal to those of $X^N_\cdot$.  Then, $\tfrac{1}{N}\sup_{|t-s|\leq \theta}|X^N_t-X^N_s| \leq \tfrac{1}{N}\sup_{|t-s|\leq \theta}|Z^N_t-Z^N_s|$.  Tightness then follows by that of scaled simple symmetric random walk on $\T_N$.

Moreover, by Aldous criterion, as
\begin{align*}
\lim_{\gamma\downarrow 0}\sup_\tau \sup_{\theta\leq \gamma}\frac{1}{N^2}\mathbb{E}^N[|M^N_{\tau+ \theta} - M^N_\tau|^2]
 = \lim_{\gamma\downarrow 0}\sup_\tau \sup_{\theta\leq \gamma}\mathbb{E}^N\big[\int_\tau^{\tau+\theta}\frac{g(\eta_r(X^N_r)}{\eta_r(X^N_r)}dr\big] 
\leq \lim_{\gamma\downarrow 0}\sup_{\theta\leq \gamma}g^*\theta = 0,
\end{align*}
we conclude that $\tfrac{1}{N}M^N_\cdot$ is tight in the Skorohod topology.
Since $\tfrac{1}{N}X^N_t - \tfrac{1}{N}X^N_0 - C^N_t = \tfrac{1}{N}M^N_t$ (cf. \eqref{eq:Xt}), any limit point $(x_\cdot, m_\cdot, C_\cdot)$ of $\big(\tfrac{1}{N}X^N_\cdot, \tfrac{1}{N}M^N_\cdot, A^N_\cdot\big)$ is such that $x_t - x_0-C_t = m_t$.  Since $x_\cdot$ and $C_\cdot$ are supported on continuous paths by the tightness shown in the uniform topology, we conclude $m_\cdot$ is also supported on continuous paths, and therefore  
$\tfrac{1}{N}M^N_\cdot$ is tight in the uniform topology.

The empirical measure $\pi^N_\cdot$ is also tight in the uniform topology as shown in Section 8 in \cite{LPSX} with respect to the hydrodynamic limit Theorem \ref{main thm} shown there.

Finally, by Mit\^{o}ma's criterion, to show $\tau_{X^N_t}\pi^N_t$ is tight, it is sufficient to show $\langle G,\tau_{X^N_t}\pi^N_t\rangle $ is tight for each continuous, compactly supported function $G$.  Write
\begin{align*}
\langle G,\tau_{X^N_t}\pi^N_t\rangle &=\frac{1}{N}\sum_{x\in \T_N} \eta_t(x+X^N_t)G(\tfrac{x}{N})= \frac{1}{N}\sum_{x\in \T_N} \eta_t(x)G(\tfrac{x -X^N_t}{N}) = \langle \tau_{-\tfrac{1}{N}X^N_t}G, \pi^N_t\rangle.
\end{align*}
Since $t\mapsto \langle \tau_{-x_t}G, \pi_t\rangle$ is a continuous function of $(x_\cdot, \pi_\cdot)$ and
$(\tfrac{1}{N}X^N_\cdot, \pi^N_\cdot)$ is tight with respect to the uniform topology, 
we have that $\tau_{X^N_\cdot}\pi^N_\cdot$ is also tight in the uniform topology.
\end{proof}

\subsection{Identification of limit points and proof of Theorem \ref{thm alpha tg}}
\label{identification subsection}
We now describe the limit points of $Q_N$, in which the fourth item is the statement of Theorem \ref{thm alpha tg}.

\begin{theorem}
Under the same assumptions as in Theorem \ref{thm alpha tg},
$Q_N$ converges with respect to the uniform topology
 to the law $Q$ concentrated on trajectories $\XX = (\pi_{\cdot}, \tau_{x_{\cdot}}\pi_{\cdot}, x_{\cdot}, m_{\cdot}, A_{\cdot})$ such that
    \begin{enumerate}
        \item $\pi_t(du) = \rho(t, u) \,du$ where $\rho$ is the unique weak solution to \eqref{eqn: with ave}.
        \item $\tau_{x_t}\pi_t(du) = \rho(t, x_t + u) \,du$.
        \item $\ds A_t = \int_0^t \frac{\Phi(\rho(s, x_s))}{\rho(s, x_s)} \,ds$.
        \item $x_t$ is the unique weak solution for \eqref{eq:thm:sde:alpha} with $\ds x_0 \stackrel{d}{=} Z_X$, governed by the process measure \eqref{process measure}.
    \end{enumerate}
    \label{thm:main_2}
\end{theorem}

\begin{proof}
    By Theorem \ref{tightness thm},
    we have tightness of $\XX_N$ with respect to the uniform topology. So we may take a subsequence $\{N_k\}$ such that $\XX_{N_k}$ converges in distribution to $\XX$ supported on continuous paths on $[0,T]$. To reduce notation, let us assume that $\XX_N$ converges in distribution to $\XX$ whose law is $Q$. Now, we will prove items \#1 - \#4.

    \paragraph*{Claim \#1.}
    The first item is the hydrodynamics result Theorem \ref{main thm} (cf. Theorem 3.3 in \cite{LPSX}). 
    
    \paragraph*{Claim \#2.}
    For the 2nd item, note as in the proof of Theorem \ref{tightness thm} for any $G \in C(\T)$, we have
            $\langle G,\, \tau_{X_t^N}\pi_t^N \rangle 
         = \langle \tau_{(-\tfrac{1}{N}X_t^N)}G,\, \pi_t^N \rangle$.
The right hand side $\langle \tau_{-\tfrac{1}{N}X_\cdot^N}G, \pi_\cdot^N\rangle$ is a continuous function of $(\tfrac{1}{N}X^N, \pi^N)$ which converges to $\langle \tau_{-x_\cdot}G, \pi_\cdot \rangle$.  To identify the limiting distribution, it is sufficient to identify the finite dimensional distributions.    In particular, since $\langle \tau_{-X_t^N}G, \pi_t^N\rangle$ converges to
      $\int_\T G(u - x_t) \,\rho(t, u) \,du
            = \int_\T G(u) \,\rho(t, x_t + u) \,du$,
we have $\tau_{\frac{1}{N}X_\cdot^N}\pi_\cdot^N \rightarrow \tau_{x_\cdot}\pi_\cdot(du) = \rho(\cdot, x_\cdot + u)\,du$ as claimed.

    \paragraph*{Claim \#3.}
    For the 3rd item, we apply Lemma \ref{replacement-lemma} to \eqref{eq:qv} by setting $h(n) = g(n)/n$, $H(\rho) = \Phi(\rho)/\rho$ and $D^N_{X^N_s}\equiv 1$. With $G=\iota_\theta$, noting $\eta^{\theta N}(x) = \langle \iota_\theta, \pi^N\rangle$
    where $\iota_\theta = (2 \theta)^{-1} \OneB_{[-\theta, \theta]}$, we have by Claim \#2
  taking $N \rightarrow \infty$ that
    \begin{equation*}
        \lim_{\ell \rightarrow \infty} \lim_{\epsilon \rightarrow 0} \lim_{\theta \rightarrow 0} ~Q\left[ ~\left| A_t - \int_0^t \frac{1}{2\epsilon} \int_{-\epsilon}^{\epsilon} \overline{H}_\ell(\langle \tau_{x_s}\iota_\theta,~ \pi_s \rangle) \,du \,ds \right| > \delta \right] = 0 ~~,~~ \forall \delta > 0 ~,~ \forall t \in [0, T].
    \end{equation*}
Since $\rho$ is continuous (Part 2(a) of Lemma \ref{lem:smoothness}) given $\alpha\in C(\T)$ and $\rho_0\in C^\beta(\T)$ for $\beta>0$, as $\theta \rightarrow 0$, we obtain $\langle \tau_{x_s} \iota_\theta, \pi_s \rangle \rightarrow \rho(s, x_s)$.
In addition, by bounded convergence, we have $\overline{H}_\ell(\rho) \rightarrow H(\rho)$ as $\ell\uparrow\infty$.
Therefore, $Q$-a.s.,
    \begin{equation*}
        \lim_{\ell \rightarrow \infty} \lim_{\epsilon \rightarrow 0} \lim_{\theta \rightarrow 0}~ \int_0^t \frac{1}{2\epsilon} \int_{-\epsilon}^{\epsilon} \overline{H}_\ell(\pi_s(\tau_{x_s} \iota_\theta) \,du \,ds = \int_0^t \frac{\Phi(\rho(s, x_s))}{\rho(s, x_s)} \,ds.
    \end{equation*}

    \paragraph*{Claim \#4.}
  The quadratic variation $\langle \tfrac{1}{N}M^N \rangle_t$ converges to $A_t$ by Claim $\#3$.  Then, by Proposition IX.1.12 in \cite{revuz2004continuous}, the martingale $\tfrac{1}{N}M_t^N$, as its quadratic variation is uniformly bounded in $N$, converges to a continuous martingale $m_t$.
  Hence, by Corollary VI.6.29 in \cite{revuz2004continuous}, $(\tfrac{1}{N}M^N_t, \langle \tfrac{1}{N}M^N\rangle_t)$ converges to $(m_t, \langle m\rangle_t)$, and as $(\tfrac{1}{N}M^N_t, \langle \tfrac{1}{N}M^N\rangle_t)$ converges to $(m_t, A_t)$, the quadratic variation of $\langle m\rangle_t=A_t$.

By Doob's martingale representation Theorem 3.4.2 and Remark 3.4.3 in \cite{karatzas1991brownian} (noting the derivative of $A_t$ is positive as $\Phi(u)/u\geq g_*>0$),
there exists a Brownian motion $B$ on the probability space $(\Omega, \FF, P)$,
where $m_t$, $x_t$, $A_t$ are defined, with the same filtration $\{\FF_t\}$,
such that $m_t = \int_0^t \sqrt{\frac{\Phi(\rho(s, x_s))}{\rho(s,x_s)}}d{B}_s$.

    The random environment factor $\alpha^N_{X_\cdot^N} = \alpha^N_{N(\tfrac{1}{N}X^N_\cdot)} = Y^N_{N(\tfrac{1}{N}X^N_\cdot)}$ in the drift term of \eqref{eq:Xt}.  As $Y^N_{N\cdot}$ converges uniformly to the continuous limit $\alpha_\cdot$, we may replace $Y^N_{N(\tfrac{1}{N}X^N_\cdot)}$ with $\alpha(\tfrac{1}{N}X^N_\cdot)$ in \eqref{eq:Xt} with a vanishing deterministic error, since $g(n)/n$ is bounded by (LG).  Note that the function $\alpha(x_\cdot)$ is continuous in $x_\cdot$.

Therefore, applying Lemma \ref{replacement-lemma} to the drift term 
$\int_0^t \alpha(\tfrac{1}{N}X^N_s) \frac{g(\eta_s(X^N_s))}{\eta_s(X^N_s)} ds$ in \eqref{eq:Xt}, with $h(n)=g(n)/n$ and $D^N_{X^N_s}=\alpha(\tfrac{1}{N}X^N_s)$, and performing the same analysis as in the proof of claim \#3 yields convergence of 
$\int_0^t  \alpha(\tfrac{1}{N}X^N_s) \frac{g(\eta_s(X^N_s))}{\eta_s(X^N_s)} ds$
to $\int_0^t \alpha(x_s) \frac{\Phi(\rho(s,x_s))}{\rho(s,x_s)} ds$, the error vanishing in probability with respect to $Q_N$.
We have also $\frac{1}{N}(X^N_t - X^N_0)$ converges in distribution to $x_t-x_0$.    By assumption, the initial distribution of $\frac{1}{N}X^N_0$ converges to the law of $Z_X$.  Therefore, $x_0$ is distributed as $Z_X$. 

Let $\hat x_t$ be the $1$-periodic extension of $x_t$.  Since $Q_N$ converges weakly to $Q$, and $\mathcal{U}_\delta=\big\{\hat x_t-\hat x_0-C_t-m_t|>\delta\big\}$ is an open set in $C([0,T])$, we conclude from the Portmanteau theorem applied to $\mathcal{U}_\delta$ that $\hat x_t-\hat x_0=C_t+m_t$.
Hence, there exists a probability space 
$(\Omega, \FF, P)$ and a Brownian motion ${B}$ defined on that probability space for which the convergence point $x_t$ satisfies the stochastic integral equation given in (\ref{eq:thm:sde}), interpreted via $\hat x_t$ in \eqref{eq:diff_torus_sde_R}.   
  So the triple $((x_t, {B}), 
	(\Omega, \FF, P), 
	\{\FF_t\})$ is a weak solution to \eqref{eq:thm:sde:alpha}.

  By the uniqueness of such weak solutions (Theorems III.3.5 (and remark after) and III.3.6 in \cite{Gihman}; see also \cite{Stroock}),
  given that the diffusion coefficient $\frac{\Phi(\rho(s,x_s))}{\rho(s,x_s)}$ is continuous, positive and bounded above and below, and the drift $\alpha(x_s)\frac{\Phi(\rho(s,x_s))}{\rho(s,x_s)}$ is continuous and bounded, all limit points converge to this solution.  
  
  We remark by the same citations, uniqueness of weak solution also holds for $\hat x_\cdot$ when the initial condition is $\hat x_0=z$ for each $z\in \R$.  
  Hence, the process measure $E_{Z_X}[\mathcal{P}^z]$ discussed in \eqref{process measure} governs the distribution of $x_\cdot$ when starting from $x_0\stackrel{d}{=} Z_{X}$.
  This completes the proof.  
    \end{proof}
		
%%%%%%%%%%%%%%%%%%%%%%%%%%%%%%%

\section{Proof of Theorem \ref{thm singular sde}: Singular diffusion limit as $\varepsilon\downarrow 0$}
\label{sec:varep-to-0}
We focus on the quenched setting in the following.  We observe that $x^\varepsilon$ on $\T$ has representation as a diffusion $\hat x^\varepsilon$ on $\R$ with periodic coefficients (cf. \eqref{eq:diff_torus_sde_R}).  
The plan is to take the limit of the It\^{o}-McKean representation of $\hat x^\varepsilon$, and thereby show Theorem \ref{thm singular sde}.

 Define $\chi_\varepsilon(t, x) = \Phi(\rho^\varepsilon(t,x))/\rho^\varepsilon(t,x)$ for $\varepsilon>0$.
 Recall the process $\hat x^\varepsilon_t$ on $\R$, with initial value $\hat x^\varepsilon_0\stackrel{d}{=}Z_W$ and
 periodic drift $b_\varepsilon(t,x) = 2\chi_\varepsilon(t,x) \, W_\varepsilon'(x)$ and diffusion coefficient $\sigma_\varepsilon^2(t,x) = \chi_\varepsilon(t,x)$, bounded above and below.  Recall also that its process measure is $E_{Z_X}[\hat{\mathcal{P}}^z]$ where $\hat{\mathcal{P}}^z$ is the process measure when $\hat x^\varepsilon_0=z\in \R$.
 
 After a few `steps', Theorem \ref{thm singular sde} is proved at the end of the section.

 \vskip .1cm
 {\it Step 1.}
 In the quenched setting, that is with respect to a realization of $W$, define the `scale' function
\begin{equation*}
        s_\varepsilon(x) = \int_0^x \exp\left(-\int_0^y 2b_\varepsilon(t, z)/\sigma^2_\varepsilon(t, z) \, dz\right)\, dy = \int_0^x \exp \left(-4 \int_0^y W_\varepsilon'(z) \,dz \right)\, dy.
\end{equation*}
  By exploiting the periodicity of $W_\varepsilon'$, we calculate
    \begin{equation}
    \label{s periodic}
        s_\varepsilon(x) = \int_0^x \exp \left(-4\left( \lfloor y \rfloor \int_0^1 W_\varepsilon'(z) \,dz + \int_0^{y - \lfloor y \rfloor} W_\varepsilon'(z) \,dz \right) \right) \,dy.
    \end{equation}
Note, as $W'_\varepsilon\in C^{1/2-}(\T)$, that $s_\varepsilon\in C^{5/2-}(\T)$.

We may identify its limit.
\begin{lemma}
    In the quenched setting, $s_\varepsilon$ converges to $s_0$, defined in \eqref{prop:s0_lim1}, uniformly on compact sets of $\R$.  Moreover, since $s_\varepsilon$ and $s_0$ are strictly increasing, $s^{-1}_\varepsilon$ also converges uniformly to $s^{-1}_0$ on compacts sets of $\R$.
    \label{prop:s0_lim}
\end{lemma}
\begin{proof}
      Since
    $W_\varepsilon'(x) = \frac{d}{dx}(\psi_\varepsilon * W)(x)$ in terms of $\psi_\varepsilon(x) = \frac{1}{\varepsilon}\psi(x/\varepsilon)$ (cf. \eqref{q_conv}), we compute
$\int_0^x W'_\varepsilon(z)dz = (\psi_\varepsilon * W)(x) - (\psi_\varepsilon * W)(0)$ whose limit is $W(x)$ as $\varepsilon\downarrow 0$, uniformly on compact sets of $\R$.
     In particular, from \eqref{s periodic}, we have
        $\lim_{\varepsilon \rightarrow 0} s_\varepsilon(x) = \int_0^x \exp \left(-4 \left\{ \lfloor y \rfloor W(1) + W(y - \lfloor y \rfloor) \right\} \right)\,dy$.
The claim holds.
\end{proof}

\vskip .1cm
{\it Step 2.}
We now observe that $s_\varepsilon(\hat x_t^\varepsilon)$ is a martingale by applying It\^{o}'s rule:  Recall $d\hat x_t^\varepsilon = 2\chi_\varepsilon W'_\varepsilon dt + \sqrt{\chi_\varepsilon}dB_t$.  Note that $d\langle \hat x^\varepsilon\rangle_t = \chi_\varepsilon(t,\hat x_t^\varepsilon)dt$, $s'_\varepsilon(x) = \exp\big(-4\int_0^x W'_\varepsilon(z)dz\big)$, and $s''_\varepsilon(x) = -4W'_\varepsilon(x)s'_\varepsilon(x)$.  Then, 
\begin{align*}
    ds_\varepsilon(\hat x_t^\varepsilon) &= s'_\varepsilon(\hat x_t^\varepsilon) \, d\hat x_t^\varepsilon + \frac{1}{2}s''_\varepsilon(\hat x_t^\varepsilon) \,d\langle \hat x^\varepsilon\rangle_t
    = \exp\left(-4\int_0^{\hat x_t^\varepsilon}W'_\varepsilon(z)dz\right)\sqrt{\chi_\varepsilon(t, \hat x_t^\varepsilon)} \,dB_t.
\end{align*}
Hence,
\begin{equation*}
    s_\varepsilon(\hat x_t^\varepsilon) -s_\varepsilon(\hat x_0^\varepsilon)= \int_0^t \exp \left(-4 \int_0^{\hat x_r^\varepsilon} W_\varepsilon'(z) \,dz \right) \,\sqrt{\chi_\varepsilon(r, \hat x_r^\varepsilon)} \,dB_r
\end{equation*}
with quadratic variation
\begin{align}
\label{quad var}
    \langle s_\varepsilon(\hat x^\varepsilon) \rangle_t = \int_0^t \exp \left(-8\int_0^{\hat x_r^\varepsilon} W_\varepsilon'(z) \,dz \right) \chi_\varepsilon(r, \hat x_r^\varepsilon) \,dr.
\end{align}
Clearly, $\langle s_\varepsilon(\hat x^\varepsilon)\rangle_t$ is a strictly increasing, continuous process.

\begin{lemma}
The quadratic variation $\langle s_\varepsilon(\hat x^\varepsilon) \rangle_t$, as $t\uparrow\infty$, 
increases to $\langle s_\varepsilon(\hat x^\varepsilon)\rangle_\infty = \infty$ a.s. with respect to the process measure $E_Z[\hat{\mathcal{P}}^z]$.
\end{lemma}

\begin{proof} The condition $\langle s_\varepsilon(\hat x^\varepsilon)\rangle_\infty <\infty$ implies a finite limit of the martingale, $\lim_{t\rightarrow\infty}s_\varepsilon(\hat x^\varepsilon_t) - s_\varepsilon(\hat x^\varepsilon_0)\in \R$, by Proposition IV.1.26 (see also Proposition V.1.8) in \cite{revuz2004continuous}.  But, $s_\varepsilon(\cdot)$ is continuous and strictly increasing.  Therefore, finiteness of $\langle s_\varepsilon(\hat x^\varepsilon)\rangle_\infty$ implies finiteness of $\lim_{t\rightarrow\infty}\hat x^\varepsilon_t\in \R$. 

Suppose now that $\langle s_\varepsilon(\hat x^\varepsilon)\rangle_\infty<\infty$.  Note that $W'_\varepsilon$ and the coefficient $\chi_\varepsilon$ is bounded above and below.  Then, convergence of $\hat x^\varepsilon_t$ as $t\uparrow\infty$, via the formula \eqref{quad var}, would imply that $\langle s_\varepsilon(\hat x^\varepsilon)\rangle_t$ diverges to infinity, a contradiction.  Hence, we must have $\langle s_\varepsilon(\hat x^\varepsilon)\rangle_\infty = \infty$ a.s.  
\end{proof}

\vskip .1cm
{\it Step 3.}
Now, since $s_\varepsilon(\hat x_t^\varepsilon) - s_\varepsilon(\hat x_0^\varepsilon)$ is a continuous martingale, vanishing at $t=0$, by the Dambis-Dubins-Schwarz result, Theorems V.1.6 in \cite{revuz2004continuous}, we have
\begin{equation}
\label{DDS BM}
    s_\varepsilon(\hat x_t^\varepsilon) -s_\varepsilon(\hat x_0^\varepsilon)= B^\varepsilon_{ \langle s_\varepsilon(\hat x^\varepsilon) \rangle_t},
\end{equation}
for $t\geq 0$, where $B^\varepsilon$ is the `DDS' standard
Brownian motion.

Let $T_\varepsilon(t)$ be the inverse
of the quadratic variation process $\langle s_\varepsilon(\hat x^\varepsilon)\rangle_t$, for $t\geq 0$.  Note that $d\langle s_\varepsilon(\hat x^\varepsilon)\rangle_t =
\exp\big(-8\int_0^{\hat x^\varepsilon_t} W'_\varepsilon(z)dz\big) \chi_\varepsilon(t, \hat x^\varepsilon_t) dt$ and
$\langle s(\hat x^\varepsilon)\rangle_{T_\varepsilon(t)}' T'_\varepsilon(t) = 1$.  Then,
    via \eqref{DDS BM}, 
		we have
\begin{align}
T_\varepsilon(t)    &= \int_0^t \exp\left(8\int_0^{s_\varepsilon^{-1}(s_\varepsilon(\hat x^\varepsilon_0)+B^\varepsilon_r)} W_\varepsilon'(z) \,dz \right) \frac{1}{\chi_\varepsilon(r, s_\varepsilon^{-1}(s_\varepsilon(\hat x^\varepsilon_0)+B^\varepsilon_r))} \,dr,
    \label{eq:T_eps_timechg}
\end{align}
a strictly increasing and onto function on $[0,\infty)$.

\vskip .1cm
{\it Step 4.}
We now consider the case $Z_X=z\in \R$ is a constant.

\begin{lemma}
    The process $\hat x_t^{\varepsilon,z} = s_\varepsilon^{-1}\big(s_\varepsilon(z) + B^\varepsilon(T_\varepsilon^{-1}(t))\big)$ on $\R$ 
    converges in distribution to the process $\hat x_t^{0,z}= s_0^{-1}\big(s_0(z)+B^0(T_0^{-1}(t))\big)$, defined near \eqref{prop:s0_lim1}, 
    in probability with respect to $\PB_W$.
    \label{prop:diff_limit_on_R}
\end{lemma}
\begin{proof}
We prove convergence in $\PB_W$-probability of $\hat x_t^{\varepsilon,z}$ to $\hat x_t^{0,z}$ in distribution
by showing that any subsequence $\hat x_t^{\varepsilon_n}$ has a further subsequence converging in distribution to $\hat x_t^0$, almost surely with respect to $\PB_W$.   In other words, let $F:C([0,T])\rightarrow \R$ be a bounded continuous function and $\PB_B$ be the law of a Brownian motion $B$.  Since, given the realization $W$, the laws of $x^\varepsilon$ and $x^0$ can be written in terms of a Brownian motion, we show that any subsequence of $E_{\PB_{B^\varepsilon}}[F(x^\varepsilon)]=E_{\PB_B}[F(x^\varepsilon)]$ has a further subsequence which converges to $E_{\PB_{B^0}}[F(x^0)]=E_{\PB_B}[F(x^0)],$ a.s. $\PB_W$.

We may also view $T_\varepsilon$ and $T_0$ as functions of a Brownian path.
In this sense, it suffices to show with respect to a realization $B$ a.s. $\PB_B$ that there exists a subsequence $\bar{\varepsilon}_k := \varepsilon_{n_k}$ such that $T_{\bar{\varepsilon}_k} \rightarrow T_0$  uniformly on $[0, T^\#]$, a.s.-$\PB_W$ for each $T^\#>0$. For, if this holds, then there is a common subsequence $\bar{\varepsilon}_k$ such that $T_{\bar{\varepsilon}_k} \rightarrow T_0$  uniformly on $[0, N]$ for each $N\in \N$ a.s.-$\PB_W$.  Then, it follows $T_{\bar{\varepsilon}_k}^{-1} \rightarrow T_0^{-1}$ uniformly on $[0, T]$, a.s.-$\PB_W$. Combining with the continuity of the Brownian path $B$ and the uniform convergence of $s_\varepsilon \rightarrow s_0$ and $s_\varepsilon^{-1} \rightarrow s_0^{-1}$ on compact subsets of $\R$ (Lemma \ref{prop:s0_lim}), we may show the desired result by bounded convergence theorem,
    \begin{align*}
        \lim_{k \rightarrow \infty} E_{\PB_B}[F(\hat x^{\bar{\varepsilon}_k})]&= \lim_{k \rightarrow \infty} E_{\PB_B}[F(s_{\bar{\varepsilon}_k}^{-1}(s_{\bar{\varepsilon}_k}(z)+B(T_{\bar{\varepsilon}_k}^{-1}(\cdot))))] \\
        &= E_{\PB_B}[F(s_0^{-1}(s_0(z)+B(T_0^{-1}(\cdot))))] = E_{\PB_B}[F(\hat x^0)], 
         \PB_W-\text{a.s.}
    \end{align*}

    To show there is a sequence $\bar{\varepsilon}_k$ such that $T_{\bar{\varepsilon}_k}$ converges to $T_0$ uniformly on $[0, T^\#]$, notice that by the continuity of $B$, the compactness of $[0, T^\#]$, and the uniform convergence of $s_\varepsilon$ and $s^{-1}_\varepsilon$ on compact subsets, the first factor in the integrand of $T_\varepsilon(t)$ in \eqref{eq:T_eps_timechg} converges uniformly on $[0, T^\#]$:  That is, by periodicity of $W'_\varepsilon$,
    \begin{equation*}
        \lim_{\varepsilon \rightarrow 0} \int_0^{s_\varepsilon^{-1}(s_\varepsilon(z)+B_r)} W_\varepsilon'(z) \,dz = \lfloor s_0^{-1}(s_0(z)+B_r) \rfloor W(1) + W(s_0^{-1}(s_0(z)+B_r) - \lfloor s_0^{-1}(s_0(z)+B_r) \rfloor).
    \end{equation*}
 In addition, since $\Phi(u)/u$ is smooth and bounded below and above (cf. Lemma \ref{lem:smoothness} with time horizon $T^\#$), the uniform convergence of a subsequence $\chi_{\bar{\varepsilon}_k}$ to $\chi_0$ and therefore $1/\chi_{\bar{\varepsilon}_k}$ to $1/\chi_{0}$ 
    follows from 
    $\rho^{\varepsilon} \rightarrow \rho^0$ in $\LL_{T^\#}^{\alpha-1}$ where $\alpha \in (13/9, 3/2)$ in probability w.r.t $\PB_W$ globally in time.
			
			Indeed, first, by Lemma \ref{lem:smoothness}, given the assumption $\rho_0\in C^{1+\beta}(\T)$ for $\beta>0$, we are assured that $\rho^\varepsilon\in \LL_{T^\#}^{\alpha -1}$ is the gradient of the unique para-controlled solution of \eqref{u-equation}.  Second, 
       recalling the definition of $\LL_{T^\#}^{\alpha-1}$, the convergence in probability in Theorem \ref{singular hyd thm} (with time horizon $T^\#$) allows to choose a subsequence $\bar{\varepsilon}_k$ so that $\rho^{\bar{\varepsilon}_k} \rightarrow \rho^0$ uniformly in $[0,T^\#]$ a.s.-$\PB_W$.
    Then, $|~ \chi_{\bar{\varepsilon}_k}(t, s_{\bar{\varepsilon}_k}^{-1}(s_{\bar{\varepsilon}_k}(z)+B_t)) - \chi_0(t, s_0^{-1}(s_0(z)+B_t)) ~|$ is bounded above by
    \begin{equation*}
         |~ \chi_{\bar{\varepsilon}_k}(t, s_{\bar{\varepsilon}_k}^{-1}(s_{\bar{\varepsilon}_k}(z)+B_t)) - \chi_{\bar{\varepsilon}_k}(t, s_0^{-1}(s_0(z)+B_t)) ~| ~+~ |~ \chi_{\bar{\varepsilon}_k}(t, s_0^{-1}(s_0(z)+B_t)) - \chi_0(t, s_0^{-1}(s_0(z)+B_t)) ~|.
    \end{equation*}

    Observe that the second term vanishes in the limit since $\chi_{\bar{\varepsilon}_k} \rightarrow \chi_0$ uniformly. For the first term, note that the function $\Phi(\rho)/\rho$ is Lipschitz for $\rho\in [\rho_-, \rho_+]$ since $\Phi$ is continuously differentiable.
    Therefore, we have
    \begin{align*}
      &  |~ \chi_{\bar{\varepsilon}_k}(t, s_{\bar{\varepsilon}_k}^{-1}(s_{\bar{\varepsilon}_k}(z)+B_t)) - \chi_{\bar{\varepsilon}_k}(t, s_0^{-1}(s_0(z)+B_t)) ~| \\
      &\quad\quad\leq~ \|\chi_{\bar{\varepsilon}_k}\|_{Lip}
        ~| \rho^{\bar{\varepsilon}_k}(t, s_{\bar{\varepsilon}_k}^{-1}(s_{\bar{\varepsilon}_k}(z)+B_t)) - \rho^{\bar{\varepsilon}_k}(t, s_0^{-1}(s_0(z)+B_t)) |.
    \end{align*}
 We may bound the second factor on the right hand side as follows:
    \begin{align*}
       & | \rho^{\bar{\varepsilon}_k}(t, s_{\bar{\varepsilon}_k}^{-1}(s_{\bar{\varepsilon}_k}(z)+B_t)) - \rho^{\bar{\varepsilon}_k}(t, s_0^{-1}(s_0(z)+B_t)) |\\
        &\quad\quad\leq~
        \,[\rho^{\bar \varepsilon}]_{\alpha-1} \sup_{t \in [0, T]} |  s_{\bar{\varepsilon}_k}^{-1}(s_{\bar{\varepsilon}_k}(z)+B_t) - s_0^{-1}(s_0(z)+B_t) |^{\beta}
    \end{align*}
    where 
    $[\rho^{\bar\varepsilon}]_{\alpha -1}$ is the H\"{o}lder constant of $\rho^{\bar\varepsilon}$. Since
    $\rho^{\bar{\varepsilon}_k}$ converges to $\rho^0$ in $\mathcal{L}^{\alpha-1}_{T^\#}=C_{T^\#}C^{\alpha -1}\cap C_{T^\#}^{(\alpha -1)/2}L^\infty$ (note $\hat C^{\alpha -1} = C^{\alpha -1}$ as $\alpha \in (13/9, 3/2)$; see remark after Lemma \ref{lem:smoothness}), 
 we have $[\rho^{\bar\varepsilon}]_{\alpha -1}$ converges to $[\rho^0]_{\alpha -1}<\infty$. Then, given $s_\varepsilon$, $s_\varepsilon^{-1}$ converges to $s_0$, $s_0^{-1}$ uniformly on compact subsets of $\R$ (Lemma \ref{prop:s0_lim}), the right hand side vanishes. 
\end{proof}

\noindent 
{\bf Proof of Theorem \ref{thm singular sde}.}
Consider the definition $x^{\varepsilon, Z_X}_t = \hat x^{\varepsilon, Z_X}_t - \lfloor \hat x^{\varepsilon, Z_X}_t\rfloor$.  Note that the law of $x^\varepsilon_0$ is the law of $Z_X$ for each $\varepsilon\geq 0$.  Note also that $\lfloor x\rfloor:\R\rightarrow \T$ is a continuous function.  Recall the equivalence of a sequence converging in $\PB_W$-probability with every subsequence having a further subsequence converging $\PB_W$-a.s.

We have $E[F(x^{\varepsilon, Z_X})] = E_{Z_X}[E_{\PB_B}[F(\hat x^{\varepsilon,z} -\lfloor x^{\varepsilon,z}\rfloor)]]$ for a bounded, continuous function $F:C([0,T], \T)\rightarrow\infty$.
    The inner expectation converges to $G(z;W)=E_{\PB_B}[F(\hat x^{0,z} - \lfloor x^{0,z}\rfloor)]$ in $\PB_W$-probability by Lemma \ref{prop:diff_limit_on_R}.  Since $G$ is bounded, by the equivalence in terms of subsequences, by bounded convergence, we conclude the full expectation converges to $E_{Z_X}[G(z;W)]$ in $\PB_W$-probability.  Hence, the desired quenched limit is found.  As remarked after Theorem \ref{thm singular sde}, the annealed part follows as a consequence.
\qed

%%%%%%%%%%%%%%%%%%%%%%%%%%%%%%%%%%%%%%%%%%%

\section{Completion of the proof of Lemma \ref{replacement-lemma}}
We supply the needed estimates of the three terms on the right-hand side of \eqref{eq:lem:repl:split} in the following three subsections to finish the proof of Lemma \ref{replacement-lemma}.  The scheme has some similarity with that in \cite{JLS}, \cite{JLS1}, although there are
many differences here given the context of the inhomogeneous environment.  In this regard, estimates from \cite{LPSX} will be useful.

\subsection{Local 1-block}
\label{sec:loc_1b}

In the proof of Lemma \ref{replacement-lemma}, the local 1-block deals with the first term in \eqref{eq:lem:repl:split}. In this term, $h(\eta_t(X_t^N))$ is replaced by an average quantity $H(\eta_t^\ell(X_t^N))$ indexed by variables in the local $\ell$-neighborhood 
of $X_t^N \in \T_N$. 
We use a Rayleigh-type estimation of a
variational eigenvalue expression derived from a Feynman-Kac bound.  Dirichlet forms and spectral
gap bounds play an important role.  

\subsubsection{Dirichlet forms}
\label{DF section}
Recall the generator $\mathcal L_N$, cf.\,\eqref{eq:gen}, and the
invariant measure $\nu^N$ (cf. \eqref{nu N def}, where
$\phi_{k,N}$ is taken so that $\max_k \phi_{k,N}=1$.  As $\nu^N$ is
not reversible with respect to $\mathcal L_N$, we will work with $\mathcal S_N = (\mathcal{L}_N + \mathcal{L}^*_N)/2$, the
symmetric part of $\mathcal L_N$:
\begin{align*}
\mathcal{S}_N f(x,\eta) &= \dfrac12 \, 
\sum_{k\neq x}\sum_\pm
 g(\eta(k))
\mf p_{k,\pm}^N
\big(f(x,\eta^{k,k\pm 1}) - f(x,\eta) \big)\\
&\quad +\frac{1}{2}\sum_\pm \Big\{
g(\eta(x))\frac{\eta(x)-1}{\eta(x)}
\mf p_{x,\pm}^N
\big(f(x,\eta^{x,x\pm 1}) - f(x,\eta) \big)\\
&\quad +\frac{1}{2}\frac{g(\eta(x))}{\eta(x)}
\mf p_{x,\pm}^N
\big(f(x\pm 1,\eta^{x,x\pm 1}) - f(x, \eta) \big)
\Big\}.
\end{align*}
where
$\mf p_{k,\pm}^N:=\big( \dfrac12 \pm \dfrac{\alpha_k^N}{ N}\big)
+
\dfrac{\phi_{k+1,N}}{\phi_{k,N}}
\big( \dfrac12 \mp \dfrac{\alpha_{k+1}^N}{ N}\big)$.

Then, $\nu^N$ is reversible with respect to the generator $\mathcal{S}_N$.  
Recall $\nu^{env,x}$ (cf. near \eqref{nu N def}).  Similarly, $\nu^{env,x}$ is reversible under $\mathcal{S}^{env,x}_N = (\mathcal{L}^{env,x}_N + \mathcal{L}^{*,env,x}_N)/2$, dropping the terms in $\mathcal{S}_N$ where the tagged particle moves.

The Dirichlet form $\DD(f)$ with respect to $\mathcal{S}_N$ is
\begin{align}
E_{\nu^N}
\left[ f(-\mathcal{S}_N f) \right]
&=
\dfrac12 \, 
 \sum_{k\neq x}E_{\nu^N} 
\Big[
g(\eta(k))
\mf p_{k,+}^N
\big(f(x,\eta^{k,k+1}) - f(x,\eta) \big)^2
\Big]\nonumber\\
&+
\dfrac12 \, 
E_{\nu^N}
 \Big[
g(\eta(x))\frac{\eta(x)-1}{\eta(x)}
\mf p_{x,+}^N
\big(f(x,\eta^{x,x+1}) - f(x,\eta) \big)^2
\Big]\nonumber\\
&+
\dfrac12 \, 
E_{\nu^N}
 \Big[
\frac{g(\eta(x))}{\eta(x)}
\mf p_{x,+}^N
\big(f(x+1,\eta^{x,x+1}) - f(x,\eta) \big)^2
\Big],
\label{two star reversible}
\end{align}
while the Dirichlet form $\DD^{env,x}(f)$ with respect to $\mathcal{S}^{env,x}_N$ equals
\begin{align*}
E_{\nu^{env,x}}
\left[ f(-\mathcal{S}^{env, x}_N f) \right]
&=
\dfrac12 \, 
 \sum_{k\neq x}E_{\nu^{env,x}} 
\Big[
g(\eta(k))
\mf p_{k,+}^N
\big(f(x,\eta^{k,k+1}) - f(x,\eta) \big)^2
\Big]\nonumber\\
&+
\dfrac12 \, 
E_{\nu^{env,x}}
\Big[
g(\eta(x))\frac{\eta(x)-1}{\eta(x)}
\mf p_{x,+}^N
\big(f(x,\eta^{x,x+1}) - f(x,\eta) \big)^2
\Big].
\end{align*}

We comment if $f$ did not depend on $\eta(x-1), \eta(x), \eta(x+1)$ then $\DD^{x,env}(f) = \langle f, -Lf\rangle_{\Pam_N}$, reducing to the Dirichlet form with respect to the standard process (cf. \eqref{eqn: generator L}).

\subsubsection{Spectral gap bound}
\label{spec_gap section}
For $k\in\T_N$ and $l\geq 1$, denote $\L_{k,l}=\left\{k-l,k-l+1,\ldots,k+l\right\}\subset \T_N$. Consider the two processes restricted to $\L_{k,l}$ generated by
$S_{k,l}$ and $S^{env}_{k,l}$ where
\begin{align*} 
S_{k,l}f(\eta)
&=
\dfrac12 \, \sum_{x,x+1\in \L_{k,l}}
\Big\{
g(\eta(x))
\mf p_{x,+}^N
\big(f(\eta^{x,x+1}) - f(\eta) \big) 
+
g(\eta(x+1))
\mf p_{x+1,-}^N
\big(f(\eta^{x+1,x}) - f(\eta) \big)
\Big\}.\nonumber\\
S^{env}_{k,l}f(\eta)
&=
\dfrac12 \, \sum_{\stackrel{x\in \L_{k,l}\setminus\{k\}}{x+1\in \L_{k,l}}}\Big\{
g(\eta(x))
\mf p_{x,+}^N
\big(f(\eta^{x,x +1}) - f(\eta) \big) 
 + g(\eta(x+1))\mf p_{x+1,-}^N \big(f(\eta^{x+1, x})-f(\eta)\big)\Big\}\nonumber\\
&\quad\quad +\dfrac12 \sum_{\pm}
g(\eta(k))\frac{\eta(k)-1}{\eta(k)}
\mf p_{k,\pm}^N
\big(f(\eta^{k,k\pm 1}) - f(\eta) \big).
\end{align*}
Let also $S^0_{k,l}$ and $S^{0,env}_{k,l}$ be these generators when $\alpha^N_\cdot\equiv 0$.

Let $\Omega_{k,l} = \N_0^{\L_{k,l}}$ be the state space of
configurations restricted on sites $\L_{k,l}$. For each
$\eta\in \Omega_{k,l}$, define
$\kappa_{k,l} (\eta) = \prod_{x\in \L_{k,l}}
\Pb_{\phi_{x,N}}(\eta(x))$, that is, $\kappa_{k,l}$ is the product
measure $\Pam_N$ restricted to $\Omega_{k,l}$.   Define also $\kappa_{k,l}^{env}(\eta)=\frac{\eta(k)}{\rho_{k,N}}\kappa_{k,l}(\eta)$, the measure $\nu^{env, k}$ (which conditions $X=k$), restricted to $\eta\in \Omega_{k,l}$.
Define also $\kappa^0_{k,l}$ and $\kappa^{0, env}_{k,l}$ to be the measures $\kappa_{k,l}$ and $\kappa^{env}_{k,l}$ when $\alpha^N_\cdot\equiv 0$.    
Let $\rho^*\equiv \rho_{k,N}$ when $\alpha^N_\cdot = 0$ (corresponding to $\phi=1$).

Let $\Omega_{k,l,j} = \{\eta \in \Omega_{k,l}: \sum_{x\in \L_{k,l} }\eta(x) = j\}$ be the state space
of configurations with exactly $j$ particles on the sites $\L_{k,l}$.
Let $\kappa_{k,l,j}$ and $\kappa_{k,l,j}^{env}$ be the associated reversible canonical measures
obtained by conditioning $\kappa_{k,l}$ and $\kappa_{k,l}^{env}$ on $\Omega_{k,l}$, and $\kappa^0_{k,l,j}$ and $\kappa^{0,env}_{k,l,j}$ these measures when $\alpha^N_\cdot\equiv 0$.

The corresponding Dirichlet forms $\mathcal{D}_{k,l}$, $\mathcal{D}_{k,l,j}$, and $\mathcal{D}^{env}_{k,l}$, $\mathcal{D}^{env}_{k,l,j}$ are given respectively by
\begin{align*} 
E_{\bar\kappa}
\left[ f(-S_{k,l} f) \right]
&=
\dfrac12 \sum_{x,x+1\in \L_{k,l}}
 E_{\bar\kappa}
\Big[
g(\eta(x))
\mf p_{x,+}^N
\big(f(\eta^{x,x+1}) - f(\eta) \big)^2
\Big], 
\end{align*}
and
\begin{align*}
E_{\bar\kappa^{env}}
\left[ f(-S^{env}_{k,l} f) \right]
&=
\dfrac12 \sum_{\stackrel{x\in \L_{k,l}\setminus\{k\}}{x+1\in \L_{k,l}}}
 E_{\bar\kappa^{env}}
\Big[
g(\eta(x))
\mf p_{x,+}^N
\big(f(\eta^{x,x+1}) - f(\eta) \big)^2
\Big] \nonumber\\
&\ \ \ \ + \dfrac12 E_{\bar\kappa^{env}}
\Big[
g(\eta(k))\frac{\eta(k)-1}{\eta(k)}
\mf p_{k,+}^N
\big(f(\eta^{k,k+1}) - f(\eta) \big)^2
\Big],  \nonumber
\end{align*}
with $\bar \kappa$ equal to $\kappa_{k,l}$ and $\kappa_{k,l,j}$, and $\bar\kappa^{env}$ equal to $\kappa^{env}_{k,l}$ and $\kappa^{env}_{k,l,j}$.

We will obtain spectral gap estimates for the localized inhomogeneous processes by comparison with the spectral gap for the translation-invariant localized process generated by $S^0_{k,l}$.
For $j\geq 1$, let $b_{l,j}$ and $b^{env}_{l,j}$ (which does not depend on the location $k$) be the reciprocals of the spectral gaps of $-S^0_{k,l}$ and $-S^{0,env}_{k,l}$ on $\Omega_{k,l,j}$ (cf.\,p.~374, \cite{KL}):
\begin{equation} \label {def of b lj}
b_{l,j}:=
\inf_{f} 
\dfrac
 {E_{\kappa^0_{k,l,j}}[ f(-S^0_{k,l}f)]}
 {{\rm Var}_{\kappa^0_{k,l,j}} (f)} \ \ {\rm and \ \ }b^{env}_{l,j}:=
\inf_{f} 
\dfrac
 {E_{\kappa^{0,env}_{k,l,j}}[ f(-S^{0,env}_{k,l}f)]}
 {{\rm Var}_{\kappa^{0,env}_{k,l,j}} (f)}.
\end{equation}
As $\Omega_{k,l,j}$ is a finite space, the infimum in the above formula is taken over all functions $f$ from $\Omega_{k,l,j}$ to $\R$. For all $l, j\geq 1$, we have $b_{l,j}>0$. 
By Lemma 6.2 in \cite{JLS} (or Lemma 3.1 in \cite{JLS1}), as $g_*\leq g(n)/n\leq g^*$ is bounded above and below by assumptions (LG) and (M), we have
\begin{align*}
b^{env}_{l,j} \leq (g^*g_*^{-1})^2b_{l,j}.
\end{align*}

\subsubsection{Inhomogeneous comparison}
We now state estimates which quantify expressions with respect to the inhomogeneous $\kappa^{env}_{k,l,j}$ to those under $\nu^{0,env}_{l,k,j}$.  The following lemma is the counterpart of Lemma 6.1 in \cite{LPSX} when the measures are in terms of the `environment' process where the tagged particle location is fixed.

Recall $\mf p_{k,+}^N$ in Section \ref{DF section}.
Let $ r_{k,l,N}^{-1} := \min_{x\in \L_{k,l}} \left\{ \mf p_{x,+}^N \right\}$.
\begin{lemma} \label {spectral gap ln k}
We have the following estimates:
\begin{enumerate}
\item Uniform bound:  For all $\eta \in \Omega_{k,l,j}$, we have
\begin{equation*} 
\left(\dfrac{\phi_{{\rm min},k,l}}{\phi_{{\rm max},k,l}} \right)^{2j}
\leq
\dfrac{\kappa^{env}_{k,l,j} (\eta) }{\kappa^{0,env}_{k,l,j}(\eta)}
\leq
\left(\dfrac{\phi_{{\rm max},k,l}}{\phi_{{\rm min},k,l}} \right)^{2j}
\end{equation*}
where 
$\phi_{{\rm min},k,l} = \min_{x\in\L_{k,l}} \phi_{x,N}$ and $\phi_{{\rm max},k,l} = \max_{x\in\L_{k,l}} \phi_{x,N}$.
\item Poincar\'e inequality:  We have for $j\geq 1$,
\begin{equation*}
{\rm Var}_{\kappa^{env}_{k,l,j}} (f)
\leq
 C_{k,l,j} 
 E_{\kappa^{env}_{k,l,j}}
 \left[ f(-S^{env}_{k,l}f) \right]
\end{equation*}
where 
$
 C_{k,l,j}
:=
(g^*g_*^{-1})^2b_{l,j}^{-1} r_{k,l,N}  \left(\dfrac{\phi_{{\rm max},k,l}}{\phi_{{\rm min},k,l}} \right)^{4j}
$ bounds the inverse of the spectral gap of $-S^{env}_{k,l}$ on $\Omega_{k,l,j}$.

\item For each $l$ fixed and $C=C(l)>0$, we have
$$\lim_{N\uparrow\infty} \sup_{1\leq k\leq N}\sup_{j\leq C\log(N)} \left(\dfrac{\phi_{{\rm max},k,l}}{\phi_{{\rm min},k,l}}\right)^{4j} = 1, 
\quad
\lim_{N\uparrow\infty} \sup_{1\leq k\leq N} r_{k,l,N} = 1$$
and hence, given \eqref{spec_gap_condition}, for fixed $l$ and $C=C(l)>0$, we have 
$$\lim_{N\rightarrow\infty}\sup_{1\leq k\leq N} \sup_{1\leq j\leq C\log(N)}\frac{1}{N}C_{k,l,j}=0.$$
\end{enumerate}
\end{lemma}

\begin{proof}
For part (1), write
\begin{align*}
\kappa^{env}_{k,l,j}(\eta) =  \frac{\kappa^{env}_{k,l}(\eta)}{\kappa^{env}_{k,l}(\Omega_{k,l,j})} \ \ {\rm and \ \ }
\kappa^{0,env}_{k,l,j}(\eta) = \frac{\kappa^{0,env}_{k,l}(\eta)}{\kappa^{0,env}_{k,l}(\Omega_{k,l,j})}.
\end{align*}
We may write
\begin{align*}
\kappa^{env}_{k,l}(\eta) = \frac{\eta(k)}{\rho_{k,N}}\kappa_{k,l}(\eta) \ \ {\rm and \ \ }
\kappa^{0,env}_{k,l}(\eta)  = \frac{\eta(k)}{\rho^*}\kappa^0_{k,l}(\eta).
\end{align*}
Hence, $\kappa^{env}_{k,l}(\eta)/\kappa^{0,env}_{k,l}(\eta)$ equals $\frac{\rho^*}{\rho_{k,N}}\frac{\kappa_{k,l}(\eta)}{\kappa^0_{k,l}(\eta)}$.

Via the proof of Part (1) of Lemma 6.1 in \cite{LPSX}, corresponding to the part (1) above, we have
\begin{align}
\left(\dfrac{\phi_{{\rm min},k,l}}{\phi_{{\rm max},k,l}} \right)^{2j}
 \leq \frac{\kappa_{k,l}(\eta)}{\kappa^0_{k,l}(\eta)} \leq \left(\dfrac{\phi_{{\rm max},k,l}}{\phi_{{\rm min},k,l}} \right)^{2j}
 \label{1-block-kappa}
 \end{align}
and therefore
the ratio
$\kappa^{env}_{k,l}(\eta)/\kappa^{0,env}_{k,l}(\eta)$, as $\eta(k)$ cancels,
has the same lower and upper bounds, except they are multiplied by $\rho^*/\rho_{k,N}$.
Moreover, since
$\kappa^{env}_{k,l}(\Omega_{k,l,j}) = \sum_{\eta\in \Omega_{k,l,j}}\kappa^{env}_{k,l}(\eta)$,
by writing $\kappa^{env}_{k,l}(\eta) = (\kappa^{env}_{k,l}(\eta)/\kappa^0_{k,l}(\eta)) \kappa^0_{k,l}(\eta)$, we may estimate the ratio
$\kappa^0_{k,l}(\Omega_{k,l,j})/\kappa^{env}_{k,l}(\Omega_{k,l,j})$ similarly.
Combining the above estimates straightforwardly, yields part (1).

Part (2) follows similarly, using the corresponding Part (2) of Lemma 6.1 in \cite{LPSX}.

Part (3) follows from Lemma \ref{lem: uniform bounds on phi max min}, which gives that $\frac{\phi_{max,k,l}}{\phi_{min, k,l}} = 1 + O(l/N)$, and also that $\mf p_{x,+}^N = 1 + O(1/N)$.
\end{proof}

\subsubsection{Mean zero replacement.}
In the upcoming proof of the local 1-block, we will need to replace $h(\eta(x)) - H(\eta^{l}(x))$ by a mean zero function with respect to the measure $\kappa^{env}_{k,l,j}$.

Recall that $h$ is bounded and Lipschitz.

\begin{lemma}
We have, $C=C(l)>0$ depending on $l$, that
    \begin{equation*}
       \limsup_{l \rightarrow \infty} 
\limsup_{N \rightarrow \infty} \sup_{k\in\T_N} \sup_{j \leq C \log N} \left|\, E_{\kappa^{env}_{k,l,j}}[ h(\eta(k)) ] - H(j/(2l+1)) ] \,\right| = 0.
    \end{equation*}
    \label{lem:L1b_mean_zero_func}
\end{lemma}

\begin{proof}
We now make use of the estimates in Lemma \ref{spectral gap ln k}
   to deduce
    \begin{equation*}
        \left( \frac{\phi_{min,k,l}}{\phi_{max, k,l}} \right)^{2j} E_{\kappa^{0,env}_{k,l,j}}[ h(\eta(k)) ]
        \leq
        E_{\kappa^{env}_{k,l,j}}[ h(\eta(k)) ]
        \leq
        \left( \frac{\phi_{max,k,l}}{\phi_{min,k,l}} \right)^{2j} E_{\kappa^{0,env}_{k,l,j}}[ h(\eta(k)) ].
    \end{equation*}
Note that the ratio $(\phi_{max,k,l}/\phi_{min,k,l})^{2j}$ converges to $1$ uniformly over $k\in \T_N$ and $j\leq C\log(N)$ as $N\uparrow\infty$. Since $h$ is bounded, the desired convergence follows from the convergence of 
    $$\sup_{j\geq 0}|E_{\kappa^{0,env}_{k,l,j}}[ h(\eta(k)) ] -H(j/(2l+1))|,$$ which does not depend on $k\in \T_N$, as $l\rightarrow\infty$ to $0$, as given by Lemma 6.4 in \cite{JLS} for bounded, Lipschitz $h$.
\end{proof}

%%%%%%%%%%%%%%%%%%%%%%%%%%%%%%%%%%%%%%%%%%%

\subsubsection{Local 1-block replacement}
We now give the main estimate of the section.

\begin{lemma}[Local 1-block]
    With the same notation as in Lemma \ref{replacement-lemma},
    \begin{equation*}
        \limsup_{\ell \rightarrow \infty} \limsup_{N \rightarrow \infty} ~\E^N\left[ ~\left| \int_0^t D^N_{X^N_s}\Big(h(\eta_s(X^N_s)) - H(\eta_s^{\ell}(X^N_s))\Big) \,ds \right|~ \right] = 0.
    \end{equation*}
    \label{lem:L1b}
\end{lemma}
\begin{proof}
    Denote by $V_\ell(X^N_s, \eta_s)$ the integrand appearing in the statement of this lemma. Since $h$ and $D^N_\cdot$ are uniformlly bounded, $V_\ell$ is also bounded.  By the truncation in Lemma \ref{loc trun lemma}, it suffices to show
    \begin{equation*}
        \limsup_{\ell \rightarrow \infty} \limsup_{N \rightarrow \infty} ~\E^N\left[ \int_0^t V_\ell(X^N_s, \eta_s) \OneB_{G_{N,\ell}}(X^N_s,\eta_s) \,ds \right] = 0, 
    \end{equation*}
where $G_{N,\ell} = \{(x, \eta) : \eta^\ell(x) \leq C \log N \}$.

    Applying the entropy inequality to the above expectation yields
    \begin{align*}
&            \E^N\left[ \int_0^t V_\ell(X^N_s, \eta_s) \OneB_{G_{N,\ell}}(X^N_s,\eta_s) \,ds \right]\\
            &\leq \frac{\H(\mu^N \,|\, \nu^N)}{N \gamma} +
            \frac{1}{N \gamma} \log E_{\nu^N}\left[ \exp\left(\gamma N \left| \int_0^t V_\ell(X^N_s, \eta_s) \OneB_{G_{N,\ell}}(X^N_s,\eta_s) \,ds \right| \right) \right].
    \end{align*}
By \eqref{prop:rel_ent_order_N}, the first term vanishes as $\gamma \rightarrow \infty$. It remains to estimate the second term. We may remove the absolute value in the second term by applying the inequality $e^{|x|} \leq e^x + e^{-x}$ and $\limsup \frac{1}{N}\log(a_N + b_N) \leq \max \{\limsup \frac{1}{N}\log a_N, \limsup \frac{1}{N}\log b_N\}$. 
    We need only estimate
    \begin{equation}
        \frac{1}{\gamma N} \log E_{\nu^N}\left[ \exp \left(\gamma N \int_0^t V_\ell(X^N_s, \eta_s) \, \OneB_{G_{N,\ell}}(X^N_s,\eta_s)\,ds \right) \right].
        \label{eq:lem:L1b_exp_1}
    \end{equation}

    As the time-scaled process $(X_s^N, \eta_s^N)$ has generator $N^2 \LL_N$ where $\LL_N$ is defined in \eqref{eq:gen},
we may apply Lemma 7.2 in Appendix 1 \cite{KL} to bound \eqref{eq:lem:L1b_exp_1} by
\begin{align}
\label{first eigenvalue}
        \frac{t}{\gamma N} \sup_{\|f\|_{L^2(\nu^N)} = 1} \langle (N^2 \LL + \gamma N \,V_\ell \,\OneB_{G_{N,\ell}}) \,f, f \rangle_{\nu^N}
        = t \sup_{\|f\|_{L^2(\nu^N)} = 1} \langle V_\ell\, \OneB_{G_{N,\ell}} \,f, f \rangle_{\nu^N} - \frac{N}{\gamma} \DD(f).
    \end{align}
 By first conditioning on the values of $X$, and then dropping the nonnegative `tagged particle terms' in line \eqref{two star reversible}, we may bound $\DD(f) \geq \sum_x \nu^N(X=x)\DD^{env,x}(f(x, \cdot))$.  Note that $\nu^N(X=x) = \rho_{x,N}/\|\rho_{\cdot,N}\|_{L^1(\T_N)}$.  We may divide and multiply each function $f(x,\cdot)$ by $\sqrt{E_{\nu^{env,x}}[f^2]}$, noting that $f^x= f(x,\cdot)/\sqrt{E_{\nu^{env,x}}[f^2]}$ has $L^2(\nu^{env,x})$ norm $1$.  
 
 Then, \eqref{first eigenvalue} is less than 
 \begin{align*}
 & \sup_{\|f\|_{L^2(\nu^N)} = 1} \sum_x \nu^N(X=x) E_{\nu^{env,x}}[f^2] \Big\{\langle V_\ell\, \OneB_{G_{N,\ell}} \,f^x, f^x \rangle_{\nu^{x,env}} - \frac{N}{\gamma} \DD^{env,x}(f^x)\Big\}\\
 &\leq \sup_x \sup_{\|f\|_{L^2(\nu^{x,env})}=1}   \Big\{\langle V_\ell\, \OneB_{G_{N,\ell}} \,f, f \rangle_{\nu^{x,env}} - \frac{N}{\gamma} \DD^{x, env}(f)\Big\}.
 \end{align*}  
   Since $V_\ell$ depends only on $\Omega_{x,\ell}$ in $\langle V_\ell\, \OneB_{G_{N,\ell}} \,f, f \rangle_{\nu^{env,x}}$, we may replace $f$ by its conditional expectation given $\Omega_{x,\ell}$, denoted $\hat f$, and $\nu^{env,x}$ by $\kappa^{env}_{x,\ell}$.  
 By convexity and dropping terms for jumps outside of $\Lambda_{x,\ell}$, the Dirichlet form $\DD^{env,x}(f) \geq \DD^{env}_{x,\ell}(\hat f)$.  Hence, we bound the last display by
 \begin{align}
 \label{second eigenvalue}
 &\sup_x \sup_{\|f\|_{L^2(\kappa^{env}_{x,\ell})}=1}   \Big\{\langle V_\ell\, \OneB_{G_{N,\ell}} \,f, f \rangle_{\kappa^{env}_{x,\ell}} - \frac{N}{\gamma} \DD^{env}_{x,\ell}(f)\Big\}.
 \end{align}
 
 One may now condition on the number $j$ of particles in $\Lambda_{x,\ell}$.  Because of the truncation, we may limit to $j\leq C\ell\log(N)$.  Again, for each $j$, we may multiply and divide by $\sqrt{E_{\kappa^{env}_{x,\ell,j}}[f^2]}$.  Denote by $\tilde f$ the function $f/\sqrt{E_{\kappa^{env}_{x,\ell,j}}[f^2]}$.  Since for $j\leq Cl\log(N)$, 
 we bound \eqref{second eigenvalue} by
 \begin{align}
 &\sup_x \sup_{\|f\|_{L^2(\kappa^{env}_{x,\ell})}=1}   \sum_{j\leq C\ell\log(N)} \kappa^{env}_{x,\ell}(\Omega_{x,\ell,j})E_{\kappa^{env}_{x,\ell,j}}[f^2] \Big\{\langle V_\ell \,\tilde f, \tilde f \rangle_{\kappa^{env}_{x,\ell,j}} - \frac{N}{\gamma} \DD^{env}_{x,\ell,j}(\tilde f)\Big\}\nonumber\\
&\leq\sup_x \sup_{j\leq C\ell\log(N)} \sup_{\|f\|_{L^2(\kappa^{env}_{x,\ell,j})}=1}   \Big\{\langle V_\ell \,f, f \rangle_{\kappa^{env}_{x,\ell,j}} - \frac{N}{\gamma} \DD^{env}_{x,\ell,j}(f)\Big\}.
\label{third eigenvalue}
\end{align}

By Lemma \ref{lem:L1b_mean_zero_func}, the difference $H(j/(2\ell+1)) - E_{\kappa^{env}_{x,\ell,j}}[h(\eta(0))]$ vanishes, taking note of the restrictions and order of the limits.  Therefore, we may replace $V_\ell$ by $V_{x,\ell,j} = D^N_x\big(h(\eta(x)) - E_{\kappa^{env}_{x,\ell,j}}[h(\eta(0))]\big)$, which is mean zero with respect to $\kappa^{env}_{x,\ell,j}$.  
Applying Rayleigh's expansion, Theorem 1.1 in Appendix 3 \cite{KL}, yields that \eqref{third eigenvalue} the  is bounded above by
    \begin{align*}
        \frac{(\gamma/N) \langle (-S^{env}_{x,\ell})^{-1} V_{x,\ell,j},~ V_{x,\ell,j} \rangle_{L^2(\kappa^{env}_{x,\ell,j})} }{1 - 2\| V_{x,\ell,j} \|_{\infty} (\gamma/N) C_{x,\ell,j}}
   &\leq  \frac{(\gamma/N) C_{x,\ell,j} \| V_{x,\ell,j} \|_{L^2(\kappa^{env}_{x,\ell,j})}^2 }{1 - 2 \|V_{x,\ell, j}\|_{\infty} (\gamma/N) C_{x,\ell,j} }
    \end{align*}
Since $(1/N)C_{x,\ell,j}$ vanishes as $N\uparrow\infty$ by Lemma \ref{spectral gap ln k}, and $H$ and $D^N_\cdot$ are uniformly bounded, we conclude the argument.
\end{proof}

%%%%%%%%%%%%%%%%%%%%%%%%%%%%%%%%%%%%%%%%%%%
\subsection{Local 2-block}
\label{sec:loc_2b}

 We now detail the $2$-block estimate
 following the outline of the $1$-block estimate, which will treat the second term in \eqref{eq:lem:repl:split} in the proof of Lemma \ref{replacement-lemma}.  We will be able to replace $H(\eta^\ell(x))$ by an average of $H(\eta^\ell(X^N +x))$ for $x$ in a small macroscopic $\epsilon N$ neighborhood of $X^N$.

Recall the notation $\L_{k,l}$ from
the $1$-block estimate.  
For $l\geq 1$ and $k, k'$ such that
$|k-k'|>2l$ and $k+ l \leq k'-l$, let
$\L_{k,k',l} = \L_{k,l} \cup \L_{k',l}$.  We introduce the following
localized generators $S_{k,k',l}$ and $S^{env}_{k,k',l}$ governing a process on
$\Omega_{k,k',l} = \N_0^{\L_{k,k',l}}$.  Inside each block, the
process moves as before, but we add an extra bond interaction between
sites $k+l$ and $k'-l$.  Define
\begin{align*}
S_{k,k',l}f(\eta)= S_{k,l}f(\eta) + & S_{k',l}f(\eta)
+
\dfrac12  \, g(\eta(k+l))
\mf p_{k+l,k'-l}^N 
\big(f(\eta^{k+l,k'-l}) - f(\eta) \big)\\
&+
\dfrac12 \, g(\eta(k'-l)) \, \mf p_{k'-l,k+l}^N
\big(f(\eta^{k'-l,k+l}) - f(\eta) \big)
\end{align*}
where
$\mf p_{k+l,k'-l}^N=
 \dfrac12 + \dfrac{\alpha_{k+l}^N}{ N}
+
\dfrac{\phi_{k'-l,N}}{\phi_{k+l,N}}
\big( \dfrac12 - \dfrac{\alpha_{k'-l}^N}{N}\big)$ and 
$\mf p_{k'-l,k+l}^N
=
\dfrac12 - \dfrac{\alpha_{k'-l}^N}{N}
+
\dfrac{\phi_{k+l,N}}{\phi_{k'-l,N}}
\big( \dfrac12 + \dfrac{\alpha_{k+l}^N}{N}\big)$.

Define also
\begin{align*}
S^{env}_{k,k',l}f(\eta)= S^{env}_{k,l}f(\eta) + & S_{k',l}f(\eta)
+
\dfrac12  \, g(\eta(k+l))
\mf p_{k+l,k'-l}^N 
\big(f(\eta^{k+l,k'-l}) - f(\eta) \big)\\
&+
\dfrac12 \, g(\eta(k'-l)) \, \mf p_{k'-l,k+l}^N
\big(f(\eta^{k'-l,k+l}) - f(\eta) \big).
\end{align*}

As before, consider the localized product measures
$\kappa_{k,k',l} =\kappa_{k,l}\times \kappa_{k',l}$ and 
$\kappa^{env}_{k,k',l} = \kappa^{env}_{k,l}\times \kappa_{k',l}$.
Define as
well as the canonical measures $\kappa_{k,k',l,j}$ and $\kappa^{env}_{k,k',l}$ on
$\Omega_{k,k',l,j}: = \{\eta\in \Omega_{k,k',l}: \sum_{x\in
  \L_{k,k',l}} \eta(x) = j\}$, that is $\kappa_{k,k',l}$ and $\kappa^{env}_{k,k',l}$
conditioned so that there are exactly $j$ particles counted in
$\Omega_{k,k',l}$. With the form of the rates $\mf p_{k+l, k'-l}^N$ and $\mf p_{k'-l, k+l}^N$, both sets of measures are invariant and reversible for the
dynamics with generators $S_{k,k',l}$ and $S^{env}_{k,k',l}$ respectively.

The corresponding Dirichlet forms $\mathcal{D}_{k,k',l}$, $\mathcal{D}_{k,k',l,j}$ and $\mathcal{D}^{env}_{k,k',l}(f)$, $\mathcal{D}^{env}_{k,k',l,j}$ are given as follows.  Define
\begin{align*}
\mathcal{D}_{k,k',l} &= \dfrac12  \, \sum_{x,x+1\in \L_{k,k',l}}
 E_{\kappa_{k,k',l}} \Big[
g(\eta(x)) \, \mf p_{x,+}^N
\big(f(\eta^{x,x+1}) - f(\eta) \big)^2
\Big]\nonumber\\
 &\ \ \  \ +
 \dfrac12  E_{\kappa_{k,k',l}} \Big[
g(\eta(k+l))
\mf p_{k+l,k'-l}^N
\big(f(\eta^{k+l,k'-l}) - f(\eta) \big)^2
 \Big].
\end{align*}
Similarly, the canonical form $\mathcal{D}_{k,k',l,j}$ is defined except we use the measure $\kappa_{k,k',l,j}$ instead of $\kappa_{k,k',l}$. 

 Moreover,
\begin{align*}
\mathcal{D}^{env}_{k,k',l}(f) &= \dfrac12  \, \sum_{\stackrel{x\in \L_{k,k',l}\setminus {k}}{x+1\in \Lambda_{k,k',l}} }
 E_{\kappa^{env}_{k,k',l} }\Big[
g(\eta(x)) \, \mf p_{x,+}^N
\big(f(\eta^{x,x+1}) - f(\eta) \big)^2
\Big]\\
&\ \ \ \ + \dfrac12 E_{\kappa^{env}_{k,k',l}}\Big[g(\eta(k))\frac{\eta(k)-1}{\eta(k)}\mf p^N_{k,+}\big(f(\eta^{k,k+1})-f(\eta)\big)^2\Big]\\
 &\ \ \  \ +
 \dfrac12  E_{\kappa_{k,k',l}} \Big[
g(\eta(k+l))
\mf p_{k+l,k'-l}^N
\big(f(\eta^{k+l,k'-l}) - f(\eta) \big)^2
 \Big],
\end{align*}
with the canonical form $\mathcal{D}^{env}_{k,k',l,j}$ defined analogously with $\kappa^{env}_{k,k',l,j}$ replacing $\kappa^{env}_{k,k',l,j}$.

Define $S^0_{k,k',l}$ and $S^{0,env}_{k,k',l}$ as the generators $S_{k,k',l}$ and $S^{env}_{k,k',l}$ when $\alpha^N_\cdot\equiv 0$.  
When $|k-k'|$ is large, the processes generated by $S^0_{k,k',l}$ and $S^{0,env}_{k,k',l}$ may be thought of as on adjacent blocks with a connecting bond.
 In this sense, these processes do not depend on $k,k'$, but only on the width $l$.

Let also $\kappa^0_{k,k',l}$, $\kappa^0_{k,k',l,j}$ and $\kappa^{0,env}_{k,k',l}$, $\kappa^{0,env}_{k,k',l,j}$ be the measures $\kappa_{k,k',l}$, $\kappa_{k,k',l,j}$ and $\kappa^{env}_{k,k',l}$, $\kappa^{env}_{k,k',l,j}$ when $\alpha^N_\cdot \equiv 0$.
Similar to \eqref{def of b lj}, for each $l$ and $j\geq 1$, we 
let $b_{l,l,j}, b^{env}_{l,l,j}>0$ be the inverse of the spectral gaps of $-S^0_{k,k',l}$ and $-S^{0,env}_{k,k',l}$ on $\Omega_{k,k',l,j}$:
\begin{align*}
b_{l,l,j}:=
\inf_{f} 
\dfrac
 {E_{\kappa^0_{k,k',l,j}}[ f(-S^0_{k,k',l}f)]}
 {{\rm Var}_{\kappa^0_{k,k',l,j}} (f)} \ \ {\rm and \ \ } b^{env}_{l,l,j}:=
\inf_{f} 
\dfrac
 {E_{\kappa^{0,env}_{k,k',l,j}}[ f(-S^{0,env}_{k,k',l}f)]}
 {{\rm Var}_{\kappa^{0,env}_{k,k',l,j}} (f)}.
\end{align*}
Again, we have
$b^{env}_{l,l,j} \leq (g^*g_*^{-1})^2b_{l,l,j}\leq C(g^*, g_*)l^2$,
by Lemma 6.2 in \cite{JLS} and \eqref{spec_gap_condition}, under assumptions (LG) and (M). 

Let $ r_{k,k',l,N}^{-1} :=
\min\big\{
\mf p_{k+l,k'-l}^N,
\min_{x,x+1\in \L_{k,k',l}}
\left\{
\mf p_{x,+}^N
\right\}
\big\} $.

\begin{lemma} \label {spectral gap of 2 blocks}
We have the following estimates:
\begin{enumerate}
\item Uniform bound:  For all $\eta\in \Omega_{k,k', l, j}$, we have
\begin{equation*}
\left(\dfrac{\phi_{{\rm min},k,l}}{\phi_{{\rm max},k,l}} \right)^{2j}
\left(\dfrac{\phi_{{\rm min},k',l}}{\phi_{{\rm max},k',l}} \right)^{2j}
\leq
\dfrac{\kappa^{env}_{k,k',l,j} (\eta) }{\kappa^{0,env}_{k,k',l,j}(\eta)}
\leq
\left(\dfrac{\phi_{{\rm max},k,l}}{\phi_{{\rm min},k,l}} \right)^{2j}
\left(\dfrac{\phi_{{\rm max},k',l}}{\phi_{{\rm min},k',l}} \right)^{2j}
\end{equation*}
 where we recall
$\phi_{{\rm min},z,l} = \min_{x\in\L_{z,l}} \phi_{x,N}$ and $ \phi_{{\rm max},z,l} = \max_{x\in\L_{z,l}} \phi_{x,N}$.

\item Poincar\'e inequality:
For fixed $j\geq 1$ and $k,k'$ such that $|k-k'|>2l+1$,  we have
\begin{equation*} 
{\rm Var}_{\kappa^{env}_{k,k',l,j}}(f)
\leq
C_{k,k',l,j} E_{\kappa^{env}_{k,k',l,j}} \big [ f(-S_{k,k',l}f) \big ]
\end{equation*}
where
$
C_{k,k',l,j}
\leq C(g^*, g_*)l^2 r_{k,k',l,N}\left(\dfrac{\phi_{{\rm max},k,l}}{\phi_{{\rm min},k,l}} \right)^{4j}
\left(\dfrac{\phi_{{\rm max},k',l}}{\phi_{{\rm min},k',l}} \right)^{4j}$.

\item For each $l$ fixed, and $C=C(l)>0$, we have
\[
\lim_{N\uparrow\infty}\sup_{k,k',N}\sup_{j\leq C\log(N)}\Big(\dfrac{\phi_{{\rm max},k,l}}{\phi_{{\rm min},k,l}}\Big)^{4j}  \Big(\dfrac{\phi_{{\rm max},k',l}}{\phi_{{\rm min},k',l}}\Big)^{4j} =1,
\ \ 
\limsup_{\epsilon\downarrow 0}\limsup_{N\uparrow\infty}\sup_{2l+1\leq |k'-k|\leq \epsilon N}r_{k,k',l,N} = 1.
\]
Hence, for fixed $l$ and $C=C(l)>0$, we have
$$ \limsup_{\epsilon \downarrow 0}\limsup_{N\uparrow\infty} 
 \sup_{2l+1\leq |k'-k| \leq \epsilon  N}\sup_{j\leq C\log(N)}
\epsilon C_{k,k',l,j}=0.$$
\end{enumerate}
\end{lemma}

\begin{proof}
The argument follows the proof of Lemma \ref{spectral gap ln k}, by
comparing $\kappa_{k,k',l,j}$ with $\kappa^0_{k,k',l,j}$.   However, here we have two separated intervals.  Since $\kappa_{k,k',l}$ and $\kappa^0_{k,k',l}$ factor into products indexed over these intervals, we may proceed.  Indeed, by comparing fugacities in $\Lambda_{k,l}$ with $\phi_{min, k, l}$ and $\phi_{max,k,l}$ and those in $\Lambda_{k',l}$ with $\phi_{min,k',l}$ and $\phi_{max, k', l}$, the analog of \eqref{1-block-kappa} is
$$\left(\dfrac{\phi_{{\rm min},k,l}}{\phi_{{\rm max},k,l}} \right)^{2j}
\left(\dfrac{\phi_{{\rm min},k',l}}{\phi_{{\rm max},k',l}} \right)^{2j}
\leq
\dfrac{\kappa^{env}_{k,k',l} (\eta) }{\kappa^{0,env}_{k,k',l}(\eta)}
\leq
\left(\dfrac{\phi_{{\rm max},k,l}}{\phi_{{\rm min},k,l}} \right)^{2j}
\left(\dfrac{\phi_{{\rm max},k',l}}{\phi_{{\rm min},k',l}} \right)^{2j}.$$
The rest of the argument for Part 1, and Parts 2,3 are similar as for the 1-block Lemma \ref{spectral gap ln k}.  We only note that here $\mf p^N_{k+l, k'-l}$ converges to $1$ as $N\uparrow\infty$, $\epsilon\downarrow 0$, by Lemma \ref{lem: uniform bounds on phi max min}.
\end{proof}

We state now a centering lemma, similar to Lemma \ref{lem:L1b_mean_zero_func}.  Recall that $H$ is bounded and Lipschitz, as stated in the proof of Lemma \ref{replacement-lemma} after 
\eqref{eq:lem:repl:split}.

\begin{lemma}
 \label{lem:L2b_mean_zero_func}
    We have, for $C=C(l)>0$ depending on $l$, that
    \begin{equation*}
        \limsup_{\ell \rightarrow \infty} 
		\limsup_{N \rightarrow \infty} \sup_{\stackrel{x\in\T_N}{y\in\{2\ell+1,\dots,\epsilon N\}}} \sup_{j \leq C \log N}
        \bigg| E_{\kappa^{env}_{x,x+y,l,j}}[ H(\eta^{\ell}(x)) - H(\eta^{\ell}(x + y)) ] \bigg| = 0.
    \end{equation*}
   \end{lemma}
\begin{proof}
   Since $H$ is bounded, we may reduce to the homogeneous case by using the same steps and approach as in Lemma \ref{lem:L1b_mean_zero_func}.  However, we invoke
Lemma 6.6 in \cite{JLS} to deduce the limit in the homogeneous case, when $\alpha^N_\cdot \equiv 0$.
\end{proof}

We now come to the main estimate of the section.
\begin{lemma}[Local 2-block]
    With the same notation as in Lemma \ref{replacement-lemma},
    \begin{equation*}
        \limsup_{\ell \rightarrow \infty} \limsup_{\epsilon \rightarrow 0} \limsup_{N \rightarrow \infty} ~\E^N \left[ \,\left|\, \int_{0}^t  D^N_{X^N_s}\Big(H(\eta_s^\ell(X^N_s)) - \frac{1}{\epsilon N} \sum_{x=1}^{\epsilon N} H(\eta_s^\ell(X^N_s + x)) \Big)\,ds\,\right|\, \right] = 0.
    \end{equation*}
    \label{lem:L2b}
\end{lemma}
\begin{proof}
    Let $V_{\ell, x,y}(\eta) = D^N_x\big(H(\eta^{\ell}(x)) - H(\eta^{\ell}(x + y))\big)$. Since $H$ and $D^N_\cdot$ are uniformly bounded, so is $V_{\ell,x,y}$ for all $y \in \T_N$. Therefore,
    \begin{align*}
            H(\eta^{\ell}(x)) - \frac{1}{\epsilon N} \sum_{y=1}^{\epsilon N} H(\eta^{\ell}(x + y))
            =
            \frac{1}{\epsilon N} \sum_{y=1}^{\epsilon N} V_{\ell,x,y}(\eta)
            \leq
            O\left( \frac{\ell}{\epsilon N} \right) + \frac{1}{\epsilon N} \sum_{y=2\ell+1}^{\epsilon N} V_{\ell,x,y}(\eta).
    \end{align*}

    The first term on the right hand side vanishes as $N \rightarrow \infty$. To bound the second term, 
    it is enough to show
    \begin{align*}
        \limsup_{\ell \rightarrow \infty} \limsup_{\epsilon \rightarrow 0} \limsup_{N \rightarrow \infty} \sup_{y \in \{2\ell + 1,\dots, \epsilon N\}} \E^N \left[ \,\left| \int_0^t  V_{\ell,X^N_s,y}(\eta_s) \,ds \right|\, \right] = 0.
    \end{align*}
    We may proceed as in the proof of the local 1-block Lemma \ref{lem:L1b} so that it suffices to show 
    \begin{align*}
&        \limsup_{\ell \rightarrow \infty} \limsup_{\epsilon \rightarrow 0} \limsup_{N \rightarrow \infty} \\
&\quad\quad \sup_{x\in \T_N}\sup_{y \in \{2\ell + 1,\dots, \epsilon N\} } \sup_{\|f\|_{L^2(\nu^{env,x})} = 1} \langle V_{\ell,x,y}\OneB_{G_{N,\ell}},\,f^2 \rangle_{\nu^{env,x}} - \frac{N}{\gamma} \DD^{env,x}(f) = 0
    \end{align*}
 where $G_{N,\ell} = \{(x, \eta) : \eta^{\ell}(x) +\eta^\ell(x+y)\leq C \log N\}$ and $\gamma > 0$ is a fixed constant.

 We now assert that
 \begin{align*}
        \DD^{env}_{x,x+y,l}(f) \leq C_1(1 + \epsilon N)\,\DD^{env,x}(f).
    \end{align*}
    Indeed, note that $\DD^{env}_{x,x+y,l}$ consists of the sum of Dirichlet forms $\DD^{env}_{x,l}(f)$, $\DD_{x+y,\ell}(f)$ and the Dirichlet bond from $x+\ell$ to $x+y-\ell$.  Now, Step 5 of Lemma 7.2 in \cite{LPSX} directly bounds the Dirichlet bond by $\epsilon N \DD'(f)$ where $\DD'(f)$ consists of Dirichlet bonds in $\DD^{env,x}(f)$ not involving $x$.  Therefore, $\DD'(f)\leq \DD^{env,x}(f)$ and the claim follows.

    Therefore, after localization and conditioning on the number of particles $j$, it suffices to show
    \begin{align*}
        &    \limsup_{\ell \rightarrow \infty} \limsup_{\epsilon \rightarrow 0} \limsup_{N \rightarrow \infty}
            \sup_{x\in\T_N,\,y \in \{2\ell + 1,\dots, \epsilon N\} }\sup_{j\leq C(\ell)\log(N)} \sup_{\|f\|_{L^2(\kappa^{env}_{x,x+y,\ell,j})} = 1}
            \\
        &   \quad \quad \langle V_{\ell,x,y},\,f^2 \rangle_{\nu_{\kappa^{env}_{x,x+y,\ell,j}}} - \frac{1}{2 \epsilon \gamma C_1} \DD^{env}_{x,x+y,\ell,j}(f) = 0.
    \end{align*}

We may replace $V_{\ell,x,y}$ by the centered expression $\widetilde V_{\ell, x,y} = V_{\ell,x,y}-E_{\kappa^{env}_{x,x+y,\ell,j}}[V_{\ell,x,y}]$ in the above limit, since by Lemma \ref{lem:L2b_mean_zero_func},
$E_{\kappa^{env}_{x,x+y,\ell,j}}[V_{l,x,y}]$ vanishes uniformly in the limits and restrictions given.
Then, to handle the centered $\widetilde V_{\ell,x,y}$ by Rayleigh expansion and Poincar\'e inequality, via Lemma \ref{spectral gap of 2 blocks}, we obtain
            \begin{align*}
       &    \sup_{x\in \T_N} \sup_{2\ell+1\leq y\leq \epsilon N} \sup_{j\leq C(\ell)\log(N)}\sup_{\|f\|_{L^2(\kappa^{env}_{x,x+y,\ell,j})}=1} 
            \langle \widetilde{V}_{\ell,x,y},\,f^2 \rangle_{\kappa^{env}_{x,x+y,\ell,j}} - \frac{1}{2 \epsilon \gamma C_1} \DD^{env}_{x,x+y,\ell,j}(f)\\
     &\leq        \sup_{x\in\T_N,\,y \in \{2\ell + 1,\dots, \epsilon N\} }
            \sup_{j \leq C(\ell)\log N}
            \,
            \frac{2 \epsilon \gamma C_1C_{x,x+y, \ell,j}\|\widetilde{V}_{\ell,x,y}\|^2_{L^2(\kappa^{env}_{x,x+y,\ell,j})}}{1 - 2\epsilon\gamma  C_1 C_{x,x+y,\ell,j}\|\widetilde{V}_{\ell,x,y}\|^2_{\infty} }.
    \end{align*}
		The right-hand side vanishes as $N\rightarrow\infty$ by Part 3 of Lemma \ref{spectral gap of 2 blocks}, noting that $H$ and $D^N_\cdot$ are uniformly bounded.
\end{proof}

%%%%%%%%%%%%%%%%%%%%%%%%%%%%%%%%%%%%%%%%%%%
\subsection{Global Replacement}
\label{sec:glob_repl}

We now indicate the `global' replacement to handle the third term in \eqref{eq:lem:repl:split} in the proof of Lemma \ref{replacement-lemma}.  We will be able to replace the average of $H(\eta_t^\ell(X_t^N +x))$ for $x$ in a small macroscopic $\epsilon N$ neighborhood of $X_t^N$ by the average of $\overline{H_\ell}(\eta_t^{\theta N}(X_t^N + x))$ over $x$ in the the same $\epsilon N$ neighborhood.
After some manipulation, this `global' replacement can be recovered from the proof of the hydrodynamics Lemma 5.1 in \cite{LPSX}, that is the proof of Theorem \ref{main thm}.

\begin{lemma}[Global Replacement]
    With the same notation as in Lemma \ref{replacement-lemma},
    \begin{align*}
        &\limsup_{\ell \rightarrow \infty} \limsup_{\epsilon \rightarrow 0} \limsup_{\theta\rightarrow 0} \limsup_{N \rightarrow \infty}\nonumber\\
	&\quad\quad\quad			\E^N\left[ \left| \int_0^t D^N_{X^N_s}\Big(\frac{1}{\epsilon N} \sum_{x = 1}^{\epsilon N} H(\eta_s^{\ell}(X^N_s + x)) -   \overline{H_\ell}(\eta_s^{\theta N}(X^N_s + x)) \Big)\,ds \right| \right] = 0.
    \end{align*}
    \label{lem:G}
\end{lemma}
\begin{proof}
    We can rewrite the expectation as
    \begin{equation*}
        \E^N \bigg[ ~\Big| \int_0^t D^N_{X^N_s}\Big(\frac{1}{N} \sum_{x \in \T_N} \iota_{\epsilon}(x/N) \left\{ H(\eta_s^{\ell}(X^N_s + x)) -   \overline{H_{\ell}}(\eta_s^{\theta N}(X^N_s + x)) \right\} \Big)\,ds \Big|~ \bigg]
    \end{equation*} 
where $\iota_{\epsilon}(x) = \epsilon^{-1} \OneB_{(0, \epsilon]}$.  We may change variables $X^N_s +x$ to $z$, and then limit the sum over $z$ to $|X^N_s-z|\geq 2\theta N$, as $H$ and $D^N_\cdot$ are uniformly bounded, with an error of $C(\|H\|_\infty, \sup_N\|D^N_\cdot\|_\infty)\theta$.  We need only consider
  $\E^N\left[\frac{1}{N}\sum_{|X^N_s -z|\geq 2\theta N}\left|\int_0^t D^N_{X^N_s-z}\left( H(\eta_s^{\ell}(z)) -   \overline{H_{\ell}}(\eta_s^{\theta N}(z))\right)ds \right| \right]$.

			Let
        $\VV_{\ell,\upsilon}(z, x, \eta) =  D^N_{x-z}\big( H(\eta^\ell(z)) - \overline{H_{\ell}}(\eta^\upsilon(z))\big)$.
		The desired limit will follow if we show, for fixed $\ell>0$, that
    \begin{equation}
		\label{prop:G}
        \limsup_{\theta \rightarrow 0} \limsup_{N \rightarrow \infty} \E^N \bigg[ \frac{1}{N} \sum_{|X^N_s-z|\geq 2\theta N} \Big|\int_0^t \VV_{\ell,\theta N}(z,X^N_s, \eta_s) \,ds \Big|\bigg] = 0.
    \end{equation}

The argument for \eqref{prop:G} follows that of Lemma 5.1 in \cite{LPSX}, which separates into (global) `1-block' and `2-block' estimates, namely Lemmas 6.2 and 7.2 in \cite{LPSX}.  Here, $H$ is a bounded, Lipschitz function of $\ell$ sites, where $\ell$ is fixed, whereas in \cite{LPSX}, the function dealt with there was of only one site.  Still, the same scheme holds as in the translation-invariant case in \cite{KL}; see also \cite{JLS}, \cite{JLS1} for analogous treatments.  We only mention the main steps for the (global) `1-block' as the (global) `2-block' is similar.

For the `1-block' estimate, we need to show
    \begin{equation*}
        \limsup_{\lambda \rightarrow \infty}
        \limsup_{N \rightarrow \infty}
   ~\E^N \bigg[  \frac{1}{N} \sum_{|X^N_s-z| > 2 \theta N} \Big |\int_0^t  \VV_{\ell,\lambda}(z,X^N_s, \eta_s) \,ds \Big| \bigg] = 0,
        \end{equation*}
introducing an intermediate scale $\lambda$.
As in the proof of Lemma 6.3 in \cite{LPSX}, one may introduce a truncation via the entropy inequality, as $g$ satisfies the (FEM) condition (cf. Section \ref{subsec: invariant measure}), so that one need only show 
    \begin{equation*}
        \limsup_{\lambda \rightarrow \infty}
        \limsup_{N \rightarrow \infty}
        ~\E^N \bigg[  \frac{1}{N} \sum_{|X^N_s-z| > 2 \theta N}\Big|\int_0^t  \left( \VV_{\ell,\lambda}(z,X^N_s, \eta_s) \,\OneB_{\{\eta_s^\lambda(z)\leq A\}} \right) \,ds \Big|\bigg] = 0.
    \end{equation*}

By using the entropy inequality and eigenvalue decompositions, conditioning on $x\in \T_N$, as in the `local' 1 and 2-block arguments for Lemmas \ref{lem:L1b} and \ref{lem:L2b},
it suffices to prove, for stationary inhomogeneous measures $\nu^{env,x}$ and constant $\gamma>0$, that
    \begin{align*}
       & \limsup_{\lambda \rightarrow \infty}
        \limsup_{N \rightarrow \infty}
        \sup_{x\in\T_N}
        \sup_{\|f\|_{L^2(\nu^{env,x})}=1}  \\
        &\quad\quad\quad \Big\langle \frac{1}{N}\sum_{|x-z|\geq 2\theta N} \VV_{\ell,\lambda}(z, x, \eta)\, \OneB_{\{\eta^\lambda(z)\leq A\}} )\,f, f \Big\rangle_{\nu^{env,x}} - \frac{N}{\gamma} \DD^{env,x}(f) = 0.
    \end{align*}
    
 Since $\VV_{\ell,\lambda}(x,z,\eta)\OneB_{\{\eta^\lambda(z)\leq A\}}$ does not depend on $\eta(x)$, we may condition $f$ on variables which do not depend on $\eta(x)$, the location $x$ being where the tagged particle is fixed.
 In this way, one can replace the measure $\nu^{env,x}$ with $\Pam_N$, and $\DD^{env,x}(f)$ evaluates to the Dirichlet form with respect to the standard process $\langle f,-Lf\rangle_{\Pam_N}$ (cf. remark after \eqref{two star reversible}). It would suffice to show, now uniformly over $z$, that
   \begin{align*}
       & \limsup_{\lambda \rightarrow \infty}
        \limsup_{N \rightarrow \infty}
             \sup_{x\in\T_N}\sup_{z: |x-z|\geq 2\theta N}
        \sup_{\|f\|_{L^2(\Pam_N)}=1}   \\
				&\quad\quad\quad \Big\langle  \VV_{\ell,\lambda}(z,x,\eta)\, \OneB_{\{\eta^\lambda(z)\leq A\}} \,f, f \Big\rangle_{\Pam_N} - \frac{N}{\gamma} \langle f, -Lf \rangle_{\Pam_N} = 0.
    \end{align*}
		By the proof of Lemma 6.3 (page 220) in \cite{LPSX}, one may replace $\VV_{\ell,\lambda}(z,x,\eta)$ by its `centering' $D^N_{x-z}\Big( H(\eta^\ell(z)) - E_{\kappa_{z,\lambda, (2\lambda+1)\eta^\lambda(z)}}[H(\eta^\lambda(z))]\Big)$ when $\eta^\lambda(z)\leq A$, with a uniform over $x, z$ error vanishing as $N\rightarrow\infty$.  
		
		At this point, one follows the same argument for Lemma 6.3 (page 221-222) in \cite{LPSX} to complete the argument for the (global) `1-block'.
\end{proof}

% --------------------------------------------------
% Appendices
% --------------------------------------------------
%\appendix

\end{document}